\documentclass[12pt]{amsart}

\usepackage{amsaddr}
\usepackage{amssymb}
\usepackage{extarrows}
\usepackage{geometry}
\usepackage{graphicx}
\usepackage{subcaption}
\usepackage[table]{xcolor}

\graphicspath{{pictures/}}

\newcommand{\ju}[1]{[\![{#1}]\!]}
\newcommand{\jut}[1]{[\![{#1}]\!]_{t}}
\newcommand{\vct}[1]{\mathbf{#1}}

\newcommand{\cC}{\mathcal{C}}
\newcommand{\curl}{\mathrm{curl}}
\newcommand{\cT}{\mathcal{T}}
\newcommand{\dhx}{\mathrm{d}{\vhx}}
\newcommand{\dhy}{\mathrm{d}{\vhy}}
\newcommand{\dhz}{\mathrm{d}{\vhz}}
\newcommand{\dx}{\mathrm{d}{\vx}}
\newcommand{\dvg}{\mathrm{div}}

\newcommand{\Fhin}{\mathcal{F}_{h}^{I}}
\newcommand{\Fnu}{\mathcal{F}_{\nu}}
\newcommand{\jac}{J}

\newcommand{\kk}{k}
\newcommand{\Nd}{R}
\newcommand{\Qh}{\mathcal{Q}_h}
\newcommand{\RT}{D}
\newcommand{\talpha}{\tilde{\alpha}}
\newcommand{\Th}{\mathcal{T}_h}
\newcommand{\tlambda}{\tilde{\lambda}}
\newcommand{\Tnu}{\mathcal{T}_\nu}
\newcommand{\tphi}{\tilde{\phi}}

\newcommand{\vG}{\vct{G}}
\newcommand{\vH}{\vct{H}}
\newcommand{\Vh}{\mathcal{V}_h}
\newcommand{\vHdel}{\vct{H}^{\Delta}}
\newcommand{\vhG}{\hat{\vG}}

\newcommand{\vhHdel}{\hat{\vH}^{\Delta}}

\newcommand{\vhr}{\hat{\vct{r}}}
\newcommand{\vhR}{\hat{\vct{R}}}

\newcommand{\vhx}{\hat{\vct{x}}}
\newcommand{\vhy}{\hat{\vct{y}}}
\newcommand{\vhz}{\hat{\vct{z}}}
\newcommand{\vj}{\vct{j}}
\newcommand{\vjj}{\boldsymbol{\j}}
\newcommand{\vjdel}{\vct{j}^{\Delta}}
\newcommand{\vn}{\hat{\vct{n}}}
\newcommand{\vr}{\vct{r}}
\newcommand{\vT}{\vct{T}}
\newcommand{\vtH}{\tilde{\vH}}
\newcommand{\vtHdel}{\vtH^{\Delta}}
\newcommand{\vthHdel}{\hat{\tilde{\vH}}^{\Delta}}

\newcommand{\vthjdel}{\hat{\tilde{\vjj}}^{\Delta}}

\newcommand{\vu}{\vct{u}}
\newcommand{\vv}{\vct{v}}
\newcommand{\vw}{\vct{w}}
\newcommand{\vx}{\vct{x}}

\newcommand{\vzero}{\vct{0}}

\newtheorem{thm}{Theorem}[section]

\newtheorem{lem}[thm]{Lemma}
\newtheorem{prp}[thm]{Proposition}
\newtheorem{rem}[thm]{Remark}

\definecolor{ilariablue}{rgb}{0 0.4 0.85}
\definecolor{ilariared}{rgb}{0.85 0.4 0}

\definecolor{sjoerdgreen}{rgb}{0 0.7 0.2}
\definecolor{sjoerdblue}{rgb}{0 0 0.7}

\title[An a posteriori error estimator for N\'ed\'elec
elements]{A polynomial-degree-robust a posteriori error estimator for N\'ed\'elec discretizations of magnetostatic problems}
\author{Joscha Gedicke$^{1*}$, Sjoerd Geevers$^{2*}$, Ilaria
  Perugia$^{2*}$, Joachim Sch\"oberl$^{3*}$}
\address{{\small
$^1$ Institute for Numerical Simulation, University of Bonn \\
Endenicher Allee 19b, 53115 Bonn, Germany \\
$^2$ Faculty of Mathematics, University of Vienna\\
Oskar-Morgenstern-Platz 1, 1090 Vienna, Austria \\
$^3$ Institute for Analysis and Scientific Computing, Vienna
University of Technology\\
Wiedner Hauptstrasse 8-10, 1040 Vienna, Austria
}}
\thanks{*S. Geevers, I. Perugia, and J. Sch\"oberl have been funded by the
Austrian Science Fund (FWF)
through the project F~65 ``Taming Complexity in Partial Differential
Systems''. I. Perugia has also been funded by the FWF through the
project P~29197-N32. }

\begin{document}

\maketitle

\begin{abstract}
We present an equilibration-based \emph{a posteriori} error estimator for N\'ed\'elec element discretizations of the magnetostatic problem.
The estimator
is obtained by adding a gradient correction to the estimator for N\'ed\'elec elements of arbitrary degree presented in~\cite{gedicke20}. 
This new estimator is proven to be reliable, with reliability constant 1, and efficient, with an efficiency constant that is independent of the polynomial degree of the approximation.
These properties are demonstrated in a series of numerical experiments on three-dimensional test problems.
\end{abstract}
\medskip

{\footnotesize
}
\medskip

\section{Introduction}
Magnetostatic equations of the form $\nabla\times(\mu^{-1}\nabla\times\vu) = \vj$ are often approximated using N\'ed\'elec elements. 
To control the error of the N\'ed\'elec finite element approximation,
a wide variety of \textit{a posteriori} error estimators are
available, including residual-type error estimators \cite{monk98,beck00},
hierarchical error estimators \cite{beck99}, Zienkiewicz--Zhu-type
error estimators \cite{nicaise05}, equilibration-based error
estimators \cite{braess08,tang13,creuse17,creuse19b}, and functional
estimates \cite{neittaanmaki10}. Of particular interest are the
localised equilibration-based error estimators, since \emph{(i)} they provide
an explicit upper bound on the error without any unknown constant
involved \cite{braess08}, \emph{(ii)} they are efficient with an efficiency
constant that is typically independent of the polynomial degree
\cite{braess09}, and \emph{(iii)} they only require solving small local problems. For an overview of these type of error estimators, see, for example, \cite{ern15} and the references therein.

The first localised equilibration-based error estimator for the magnetostatic problem
was introduced in \cite{braess08}. That estimator was designed for
N\'ed\'elec element approximations of lowest order only, and requires the solution of local problems on vertex patches. In \cite{gedicke20}, an alternative localised equilibration-based error estimator was presented that is applicable to N\'ed\'elec element approximations of arbitrary degree. That
method requires the solution of local problems on single elements, on single faces, and on small sets of nodes.
While it was proven in \cite{gedicke20} that the estimator satisfies bounds of the form 
\begin{equation*}
C_{\text{eff}}\cdot\text{estimator} \leq \text{error} \leq C_{\text{rel}}\cdot\text{estimator}
\end{equation*}
up to some higher-order data oscillation terms, with reliability
constant $C_{\text{rel}}=1$ and efficiency constant $C_{\text{eff}}>0$
independent of the mesh size, numerical experiments
showed that the efficiency constant $C_{\text{eff}}$ still mildly depends on the polynomial degree.
Recently, a localised \emph{quasi-equilibrated} error estimator was also introduced in \cite{chaumont_frelet20arXiv} that has an efficiency constant $C_{\text{eff}}$ independent of the polynomial degree, but a reliability constant $C_{\text{rel}}$ that is not explicitly known for non-convex domains.

In this paper, a new error estimator is constructed by adding a
gradient correction to the estimator of~\cite{gedicke20}, resulting in
an efficiency index that is now also independent of the polynomial
degree. The proof of reliability (Theorem~\ref{thm:reliability}) is a slight modification of the
corresponding one developed in~\cite{gedicke20}, whereas the proof of
efficiency with a constant independent of the polynomial degree
(Theorem~\ref{thm:efficiency}) is
significantly more involved. Unlike in~\cite{gedicke20}, we can no
longer rely on the efficiency of the residual error estimator, since
this error estimator is not polynomial-degree robust. Instead, the
efficiency proof is based on a new decomposition of the error and
relies on the stability property of the regularized Poincar\'e
integral operator proven in~\cite{costabel10}, and on the stable
broken $H^1$ polynomial extensions presented in \cite{ern19}.
The new estimator and its analysis are presented for the case of 
 piecewise constant magnetic permeability.
The extension to the case of piecewise smooth magnetic permeability
is discussed in Remark~\ref{rem:variablemu}.

The outline of this paper is as follows. In Section \ref{sec:intro},
the considered model problem and its
N\'ed\'elec finite element discretization
is presented. In Section \ref{sec:errEst}, the new error estimator is introduced, and the main theorems on reliability and efficiency are stated. In Section \ref{sec:analysis}, the efficiency of the estimator is proven. Numerical examples are presented in Section \ref{sec:numerics}, and the main results are summarised in Section~\ref{sec:conclusion}.

\section{Model problem and notation}\label{sec:intro}
In this section, we define the same model problem and notation as in~\cite{gedicke20}.
We consider the linear magnetostatic problem in the unknown magnetic field $\vH$:
\begin{align*}
\nabla\times\vH &= \vj &&\text{in }\Omega, \\
\nabla\cdot\mu\vH &= 0 &&\text{in }\Omega, \\
\vn\cdot\mu\vH &= 0 &&\text{on }\partial\Omega,
\end{align*}
where $\Omega\subset\mathbb{R}^3$ is an open, bounded, 
simply connected, polyhedral domain with a connected Lipschitz
boundary $\partial\Omega$
with outward pointing unit normal vector $\vn$, $\mu$ is a scalar magnetic permeability, $\vj$
a given divergence-free current density, and
$\nabla$, $\nabla\times$ and $\nabla\cdot$ denote the gradient, the curl, and the divergence
operator, respectively. We assume that
$\mu:\Omega\rightarrow\mathbb{R}^+$, and
$\mu_0\leq\mu\leq\mu_1$, for some positive constants $\mu_0$ and $\mu_1$.

In terms of a vector potential $\vu$ such that
$\vH=\mu^{-1}\nabla\times\vu$, the problem can be rewritten as the
following system:
\begin{subequations}
\label{eq:curlcurl}
\begin{align}
\nabla\times(\mu^{-1}\nabla\times \vu) &= \vj &&\text{in }\Omega, \\ 
\nabla\cdot\vu &= 0 &&\text{in }\Omega, \\
\vn\times\vu &= \vzero &&\text{on }\partial\Omega,
\end{align}
\end{subequations}
where the uniqueness of $\vu$ is imposed by the second equation (Coulomb's gauge).

Let $D\in\mathbb{R}^3$ be any given domain. We denote by
$L^2(D)^m$ the standard space of square-integrable
functions $\vu:D\rightarrow\mathbb{R}^m$ endowed with norm
$\|\vu\|_D^2:=\int_D \vu\cdot\vu \;\dx$ and inner product
$(\vu,\vw)_D=\int_D \vu\cdot\vw\;\dx$.
We also define the following functional spaces:
\begin{align*}
H^1(D) &:= \{\phi\in L^2(D) \;|\; \nabla \phi\in L^2(D)^3 \}, \\
{H_{0}^1(D)} &:= \{\phi\in H^1(D) \;|\; \phi=0 \text{ on }\partial\Omega \}, \\
H(\curl; D) &:= \{\vu\in L^2(D)^3 \;|\; \nabla\times\vu\in L^2(\Omega)^3 \}, \\
{H_{0}(\curl; D)} &:= \{\vu\in H(\curl;D) \;|\; \vn\times\vu=\vzero \text{ on }\partial\Omega \}, \\
H(\dvg; D) &:=  \{\vu\in L^2(D)^3 \;|\; \nabla\cdot\vu\in L^2(\Omega) \}, \\
H(\dvg^0; D) &:=  \{\vu\in L^2(D)^3 \;|\; \nabla\cdot\vu \equiv 0 \},\\
H_{0,\Gamma}^1(D) &:= \{\phi\in H^1(D) \;|\; \phi=0 \text{ on }\Gamma \},
\end{align*}
where, in the last definition, $\Gamma$ is any two-dimensional manifold
$\Gamma\subset\partial D$.
If $\cT_D$ is any tessellation of $D$, we set
\begin{align*}
H^1(\cT_D) &:= \{\phi\in L^2(D) \;|\; \phi|_T\in H^1(T) \text{ for all }T\in\cT_D\}, \\
H(\curl;\cT_D) &:= \{\vu\in L^2(D) \;|\; \vu|_T\in H(\curl;T) \text{ for all }T\in\cT_D\}, \\
H(\dvg;\cT_D) &:= \{\vu\in L^2(D) \;|\; \vu|_T\in H(\dvg;T) \text{ for all }T\in\cT_D\}.
\end{align*}

The variational formulation of problem \eqref{eq:curlcurl} reads as
follows: Find the vector potential $\vu\in H_0(\curl;\Omega)\cap H(\dvg^0;\Omega)$ such that
\begin{align}
\label{eq:WF}
(\mu^{-1}\nabla\times \vu,\nabla\times\vw)_{\Omega} &= (\vj,\vw)_{\Omega} &&\forall \vw\in H_0(\curl;\Omega).
\end{align}

Before introducing a finite element approximation of~\eqref{eq:WF}, we
define the following polynomial spaces. 
For any $D\subset\mathbb{R}^3$, let $P_k(D)$ denote the space of
polynomials of degree $k$ or less.
Moreover, for any tetrahedron $T$, let $\Nd_k(T)$ and $\RT_k(T)$ denote
the first-kind N\'ed\'elec space and the Raviart-Thomas space,
respectively:
\begin{align*}
\Nd_k(T) &:= \{\vu\in P_k(T)^3 \;|\; \vu(\vx) = \vv(\vx) + \vx\times\vw(\vx) \text{ for some }\vv,\vw\in P_{k-1}(T)^3\}, \\
\RT_k(T) &:= \{\vu\in P_k(T)^3 \;|\; \vu(\vx) = \vv(\vx) + \vx w(\vx) \text{ for some }\vv\in P_{k-1}(T)^3, \;\dots \\
& \qquad\qquad w\in P_{k-1}(T)\}.
\end{align*}
Furthermore,
for any domain $D\subset\mathbb{R}^3$ with a tessellation $\cT_D$,
we define the \emph{discontinuous} spaces
\begin{align*}
P_k^{-1}(\cT_D) &:= \{\phi\in L^2(D) \;|\; \phi|_T \in P_k(T) \text{ for all }T\in\cT_D \}, \\
\Nd_k^{-1}(\cT_D) &:= \{\vu\in L^2(D)^3 \;|\; \vu|_T \in \Nd_k(T) \text{ for all }T\in\cT_D \}, \\
\RT_k^{-1}(\cT_D) &:= \{\vu\in L^2(D)^3 \;|\; \vu|_T \in \RT_k(T) \text{ for all }T\in\cT_D \}, 
\end{align*}
and the \emph{conforming} spaces 
\begin{align*}
P_k(\cT_D)&:=P^{-1}_k(\Th)\cap H^1(D), &  {P_{k,0}(\cT_D)} &:=P^{-1}_k(\cT_D)\cap H^1_{0}(D),\\
\Nd_k(\cT_D)&:=\Nd^{-1}_k(\Th)\cap H(\curl;D),  & {\Nd_{k,0}(\cT_D)} &:=\Nd^{-1}_k(\cT_D)\cap H_{0}(\curl;D),\\
\RT_k(\cT_D)&:=\RT^{-1}_k(\Th)\cap H(\dvg;D).  &
\end{align*}%
For any two-dimensional manifold $\Gamma_D\subset \partial D$, also define 
\begin{align*}
P_{k,0,\Gamma_D}(\cT_D) &:=P^{-1}_k(\cT_D)\cap H^1_{0,\Gamma_D}(D).
\end{align*}

We consider the following finite element approximation
of~\eqref{eq:WF} on a tetrahedral mesh $\Th$ of $\Omega$ of
granularity $h$: Find
$\vu_h\in \Nd_{k,0}(\Th)$ such that
\begin{subequations}
\label{eq:FEM}
\begin{align}
(\mu^{-1}\nabla\times \vu_h,\nabla\times\vw)_{\Omega} &= (\vj,\vw)_{\Omega} &&\forall \vw\in \Nd_{k,0}(\Th), \label{eq:FEMa}\\
(\vu_h,\nabla\psi)_{\Omega} &=0 &&\forall \psi\in P_{k,0}(\mathcal{T}_h). \label{eq:FEMb}
\end{align}
\end{subequations}
The approximation of the magnetic field is then defined by
\begin{align*}
\vH_h:=\mu^{-1}\nabla\times\vu_h.
\end{align*}
For the well-posedness of the continuous problem~\eqref{eq:WF}, see, e.g., \cite{kikuchi89}, 
and for the $h$-convergence of the finite element method~\eqref{eq:FEM}, see, e.g., \cite[Theorems~5.9 and~5.10]{hiptmair02}.

\section{A polynomial-degree-robust {a posteriori} error estimator}\label{sec:errEst}

As in~\cite{braess08,gedicke20}, the equilibrated \emph{a posteriori} error estimator we are going to introduce
is based on the following result (\cite[Theorem 10]{braess08}, \cite[Corollary 3.3]{gedicke20}):
\begin{thm}
\label{thm:estimator}
Let $\vu$ be the solution to (\ref{eq:WF}), let $\vu_h$ be the solution of (\ref{eq:FEM}), set $\vH:=\mu^{-1}\nabla\times\vu$ and $\vH_h:=\mu^{-1}\nabla\times\vu_h$, and let $\vj_h:=\nabla\times\vH_h$ be the discrete current distribution. If $\vtHdel\in L^2(\Omega)^3$ satisfies the (residual) equilibrium condition
\begin{align}
\label{eq:vtHdel}
\nabla\times \vtHdel &= \vj-\vj_h
\end{align}
in a distributional sense (i.e. $\langle \nabla\times \vtHdel, \vw \rangle = \langle \vj-\vj_h,\vw \rangle$ for all $\vw\in \cC_0^{\infty}(\Omega)^3$, where $\langle \cdot,\cdot \rangle$ denotes the application of a distribution to a function in $\cC_0^{\infty}(\Omega)^3$), then
\begin{align}
\label{eq:errEst2}
\|\mu^{1/2}(\vH-\vH_h)\|_{\Omega} &\leq \|\mu^{1/2}\vtHdel\|_{\Omega}.
\end{align}
\end{thm}

To construct a field $\vtHdel$ that satisfies \eqref{eq:vtHdel}, we use polynomial function spaces of degree $\kk$ and make the following two assumptions:
\begin{itemize}
  \item[A1.] The magnetic permeability $\mu$ is piecewise constant and the mesh is assumed to be chosen in such a way that $\mu$ is constant within each element. 
  \item[A2.] The current density $\vj$ is in $\RT_{\kk}(\mathcal{T}_h)\cap H(\dvg^0;\Omega)$.
  \end{itemize}\medskip
  
The case of a piecewise smooth instead of a piecewise constant magnetic permeability is discussed in Remark \ref{rem:variablemu} below.
  
\begin{rem}
In case assumption A2 is not satisfied,
$\vj$ can be replaced by a suitable projection $\pi_h\vj$ such as, for instance, the standard
Raviart-Thomas interpolate
in $\RT_{\kk}(\mathcal{T}_h)$.
As observed in~\cite[Section~3.1]{gedicke20}, the error then satisfies
$\|\mu^{1/2}(\vH-\vH_h)\|_{\Omega}\le \|\mu^{1/2}\vtHdel\|_{\Omega} +
\|\mu^{1/2}(\vH-\vH^{\prime})\|_{\Omega}$, with $\vH^{\prime}$ the
solution to~\eqref{eq:WF}
with $\pi_h\vj$ instead of $\vj$.
As proven in~\cite[Appendix~A]{gedicke20}, whenever $\vj$ admits a
compactly supported extension $\vj^*\in H(\dvg;\mathbb{R}^3)\cap
H^{\kk}(\mathbb{R}^3)^3$, the term $\|\mu^{1/2}(\vH-\vH') \|_{\Omega}$
is of order $h^{\kk+1}$ and therefore of higher order than $\|
\mu^{1/2}(\vH-\vH_h) \|_{\Omega}$.
\end{rem}

For the construction of a field $\vtHdel$ satisfying \eqref{eq:vtHdel}, we proceed as in \cite{gedicke20}, but perform one additional step (Step 4).

\emph{Step 1.} We compute $\vthHdel\in \Nd^{-1}_{\kk}(\Th)$ from the datum $\vj$ and the numerical solution~$\vH_h$ by solving 
\begin{subequations}
\label{eq:vthHdel}
\begin{align}
\nabla\times\vthHdel |_T &= \vjdel_T := \vj|_T-\nabla\times\vH_h|_T, && \\
(\vthHdel,\nabla\psi)_T &= 0 &&\forall \psi\in P_{\kk}(T) \label{eq:vthHdel1b}
\end{align}
\end{subequations}
for each $T\in\Th$. 

\emph{Step 2.} For each internal face $f\in\Fhin$, let $T^+$ and $T^-$ denote the two adjacent elements, let $\vn^{\pm}$ denote the normal unit vector pointing outward of $T^{\pm}$, let $\vH^{\pm}:=\vH|_{T^{\pm}}$ denote the vector field restricted to $T^{\pm}$, let $\jut{\vH}|_f := (\vn^+\times\vH^+ + \vn^-\times\vH^-)|_{f}$ denote the tangential jump operator, and let $\nabla_f$ denote the gradient operator restricted to the face $f$. We set $\vn_f:=\vn^+|_f$ and compute $\tlambda_f\in P_{\kk}(f)$ by solving
\begin{subequations}
\label{eq:tlambda}
\begin{align}
-\vn_f\times\nabla_f\tlambda_f &= \vthjdel_f := \jut{\vH_h + \vthHdel}|_f , \\
(\tlambda_f,1)_f &=0 \label{eq:tlambda1b} 
\end{align}
\end{subequations}
for each internal face $f\in\Fhin$. 

\emph{Step 3.} Let $\Qh$ denote the set of standard Lagrangian nodes corresponding to the finite element space $P_{\kk}(\Th)$.  We compute $\tphi\in P^{-1}_{\kk}(\Th)$ by solving, for each $\vx\in\Qh$, the small set of degrees of freedom $\{\tphi_{T,\vx}\}_{T:\overline{T}\ni\vx}$ such that
\begin{subequations}
\label{eq:tphi}
\begin{align}
\tphi_{T^+,\vx} - \tphi_{T^-,\vx} &= \tlambda_{f}(\vx) &&\forall f\in\Fhin:\partial f\ni\vx, \label{eq:tphi1a} \\
\sum_{T:\overline{T}\ni\vx} \tphi_{T,\vx} &= 0, &&
\end{align}
\end{subequations}
where $\tphi_{T,\vx}$ denotes the value of $\tphi|_T$ at node $\vx$.

\emph{Step 4.} Let $\Vh$ denote the set of all mesh vertices and, for each $\nu\in\Vh$, let $\Tnu$ denote the element patch consisting of all elements adjacent to $\nu$, set $\overline{\omega_\nu}:=\bigcup_{T\in\Tnu} \overline{T}$, and set $\Gamma_\nu :=\partial\omega_\nu$ whenever $\nu$ is an interior vertex and $\Gamma_\nu:=\partial\omega_\nu\setminus\partial\Omega$ whenever $\nu$ is a vertex on the boundary $\partial\Omega$. For each vertex $\nu\in\Vh$, we compute a continuous scalar field $\talpha_\nu\in P_{\kk+1,0,\Gamma_\nu}(\Tnu)$ such that 
\begin{align}
\label{eq:talpha_nu}
(\mu\nabla\talpha_\nu,\nabla\psi)_{\omega_\nu} &=  (\mu\nabla_h(\theta_\nu\tphi),\nabla\psi)_{\omega_\nu} &&\forall \psi\in P_{\kk+1,0,\Gamma_\nu}(\Tnu),
\end{align}
where $\nabla_h$ denotes the element-wise gradient operator and $\theta_\nu$ denotes the hat function corresponding to vertex $\nu$. We then extend $\talpha_\nu$ by zero to the rest of the domain $\Omega$ and set $\talpha := \sum_{\nu\in\Vh} \talpha_\nu$.

\emph{Step 5.} We compute the field
\begin{align*}
\vtHdel = \vthHdel + \nabla_h\tphi - \nabla\talpha,
\end{align*}
and compute the error estimator
\begin{align}
\label{eq:eta}
\eta_h :=\|\mu^{1/2}\vtHdel\|_{\Omega} =  \left(\sum_{T\in \Th}\eta_T^2\right)^{1/2}, &\quad\text{where}\quad
\eta_T :=\|\mu^{1/2}\vtHdel\|_T.
\end{align}

It can be shown that the problems in Steps~1--4 are well-posed and therefore that the error estimator is well defined.

\begin{thm}[well-posedness] \label{thm:wellposed}
Let $\vu$ be the solution to (\ref{eq:WF}), let $\vu_h$ be the solution to (\ref{eq:FEM}), and set $\vH:=\mu^{-1}\nabla\times\vu$ and $\vH_h:=\mu^{-1}\nabla\times\vu_h$. Also, assume that assumptions A1 and A2 hold true. Then the problems in Steps~1--4 are all well-defined and have a unique solution.
\end{thm}
\begin{proof}
In \cite{gedicke20}, it was shown that Steps~1--3 are well-defined. Step 4 is also well-defined, since $\talpha_\nu$ is the unique discrete finite element approximation for the elliptic problem $-\nabla\cdot\mu\nabla\alpha_\nu=\nabla\cdot\mu\nabla_h(\theta_\nu\tphi)$ in $\omega_\nu$ and $\alpha_\nu|_{\Gamma_\nu}=0$. 
\end{proof}

The resulting error estimator is reliable and provides an explicit upper bound on the error, i.e. the upper bound does not involve any unknown constants.

\begin{thm}[reliability]\label{thm:reliability}
Let $\vu$ be the solution to (\ref{eq:WF}), let $\vu_h$ be the solution to (\ref{eq:FEM}), and set $\vH:=\mu^{-1}\nabla\times\vu$ and $\vH_h:=\mu^{-1}\nabla\times\vu_h$. Also, 
assume that assumptions A1 and A2 hold true. Then the problems in Steps~1--4 are all well-defined and have a unique solution. Furthermore, if $\vtHdel$ is computed by following Steps~1--5, then
\begin{align*}
\|\mu^{1/2}(\vH-\vH_h)\|_{\Omega} &\leq \|\mu^{1/2}\vtHdel \|_{\Omega}=\eta_h.
\end{align*}
\end{thm}

\begin{proof}
In \cite{gedicke20}, it was shown that Steps~1--3 result in a field $\vthHdel+\nabla_h\tphi$ that satisfies
\begin{align*}
\nabla\times(\vthHdel+\nabla_h\tphi) &= \vj-\vj_h
\end{align*}
in a distributional sense. Since $\nabla\times\nabla\talpha \equiv \vzero$, we have that
\begin{align*}
\nabla\times\vtHdel = \nabla\times(\vthHdel+\nabla_h\tphi-\nabla\talpha)=\nabla\times(\vthHdel+\nabla_h\tphi) &= \vj-\vj_h.
\end{align*}
The theorem then follows from Theorem \ref{thm:estimator}.
\end{proof}

The following theorem is the main result of this paper. It states that the error estimator is efficient and that the efficiency index is bounded by a constant that is independent of the polynomial degree.

\begin{thm}[local efficiency]
\label{thm:efficiency}
Let $\vu$ be the solution to (\ref{eq:WF}), let $\vu_h$ be the solution to (\ref{eq:FEM}), and set $\vH:=\mu^{-1}\nabla\times\vu$ and $\vH_h:=\mu^{-1}\nabla\times\vu_h$. Also, 
assume that assumptions A1 and A2 hold true. If $\vtHdel$ is computed by following Steps~1--5, then
\begin{align}
\label{eq:efficiency}
\eta_T &= \|\mu^{1/2}\vtHdel\|_T \leq C\sum_{T':\overline{T'}\cap\overline{T}\neq\emptyset} \|\mu^{1/2}(\vH-\vH_h)\|_{T'}
\end{align}
for all $T\in\Th$, where $C$ is some positive constant that depends on the magnetic permeability $\mu$ and the shape-regularity of the mesh, but not on the mesh width $h$ or the polynomial degree $\kk$.
\end{thm}

The proof of Theorem \ref{thm:efficiency} is given in the next section.

\begin{rem}
As observed in \cite[Remark~3.4]{gedicke20}, for the case $\kk=1$,
this algorithm requires solving local problems with 6 unknowns per
element in Step~1, 3 unknowns per face in Step~2, and $(\# T\in\Tnu)\approx 24$ unknowns per
vertex $\nu$ in Step~3.
The problem in the additional Step~4 involves $1+(\# e:\overline{e}\ni\nu) \approx 15$ unknowns per vertex when $\kk=1$.
\end{rem}

\begin{rem}\label{rem:variablemu}
\newcommand{\Ph}{\Pi_h^{\kk-1}}
Assume that $\mu$ is \emph{piecewise smooth}, and that the mesh is
chosen in such a way that $\mu$ is smooth within each element. The
definition of the error estimator $\eta_h$ can be extended to this
case as follows.

Define $\vH_h^\ast:=\Ph\vH_h = \Ph (\mu^{-1}\nabla\times\vu_h)$, where $\Ph$ is
the weighted $ L^2(\Omega)^3$ projection onto $P_{\kk-1}^{-1}(\Th)^3$
such that $(\mu\Ph \vH_h,\vw)_\Omega = (\mu\vH_h,\vw)_{\Omega}$ for all $\vw\in P_{\kk-1}^{-1}(\Th)^3$. Set $\vj_h:=\nabla\times \vH_h$ and
$\vj_h^*:=\nabla\times\vH_h^*$. Then, compute $\vtHdel$ such that
  $\nabla\times\vtHdel = \vj-\vj_h^*$ by following Steps 1-5 with
  $\vH_h$ replaced by $\vH_h^*$. One can prove, in a way analogous to
  \cite[Section 3.2]{gedicke20} and the proof of Theorem
  \ref{thm:reliability}, that the problems in Steps 1-4 with $\vH_h$
  replaced by $\vH_h^*$ are well-posed, and thus $\vtHdel$ is well-defined. 
Then, the new local and global error estimators are defined as
\begin{align*}
  \eta_T &:= \left(\|\mu^{1/2}\vtHdel \|_{T}^2 +\| \mu^{1/2}(\vH_h-\vH_h^\ast)\|_{T}^2\right)^{1/2},\\
\eta_h &:= \left(\|\mu^{1/2}\vtHdel \|_{\Omega}^2 +\| \mu^{1/2}(\vH_h-\vH_h^\ast)\|_{\Omega}^2\right)^{1/2} = \left(\sum_{T\in\Th} \eta_T^2\right)^{1/2}.
\end{align*}
Clearly, for piecewise constant $\mu$, the new estimators coincide with the old ones.

The reliability bound 
$\|\vH-\vH_h\|_{\Omega}\le \eta_h$ follows from Theorem~\ref{thm:estimator}
and the fact that $\vtHdel+\vH_h^*-\vH_h$ satisfies the residual equilibrium condition
\begin{align*}
\nabla\times(\vtHdel+\vH_h^*-\vH_h) = \vj-\vj_h.
\end{align*}

For the local efficiency bound, one can check that Theorem~\ref{thm:efficiency} still holds true when replacing $\vH_h$ by $\vH_h^*$. 
We then only need to prove efficiency of the additional term~$\|\mu^{1/2}(\vH_h-\vH_h^*)\|_T$ for each $T\in\Th$. We have
\[
  \begin{split}
    \|\mu^{1/2}(\vH_h-\vH_h^*)\|_{T}
    &=\|\mu^{1/2}({\mathrm I}-\Ph)\vH_h \|_{T} \\
    &\le \|\mu^{1/2}({\mathrm I}-\Ph)(\vH-\vH_h)\|_{T} + \|\mu^{1/2}({\mathrm I}-\Ph)\vH\|_{T}\\
    &\le \| \mu^{1/2}(\vH-\vH_h)\|_{T} + \| \mu^{1/2}({\mathrm I}-\Ph) \vH \|_{T},
  \end{split}
\]
for all $T\in\Th$, where ${\mathrm I}$ denotes the identity operator and where the last inequality follows from the $L^2$ stability of the weighted $L^2$ projection. 
The behaviour of the second term on the right-hand side depends on the smoothness of $\vH$ and on the mesh grading towards possible solution singularities. Therefore, it behaves similarly to the actual error $\|\mu^{1/2}(\vH-\vH_h)\|_T$. 
\end{rem}

\section{Proof of Theorem \ref{thm:efficiency}}\label{sec:analysis}

In this section, we let $\vu$, $\vu_h$, $\vH$, $\vH_h$, and $\vtHdel$ be the fields as defined in Theorem \ref{thm:efficiency} and let $\vthHdel$, $\tphi$, $\talpha_\nu$, and $\talpha$ as described in Steps~1--5 of Section \ref{sec:errEst}. We will also always let $C$ denote some positive constant that may depend on the magnetic permeability $\mu$ and the shape regularity of the mesh, but not on the mesh width $h$ or the polynomial degree $\kk$.

In Section \ref{sec:errDec}, we introduce a vector field $\vhHdel\in H(\curl;\Th)$ and scalar fields $\phi\in H^1(\Th)$ and $\alpha\in H^1(\Omega)$, and show that the error $\vHdel:=\vH-\vH_h$ can be written as $\vHdel=\vhHdel + \nabla_h\phi - \nabla\alpha$. We also show there that
\begin{subequations}
\label{eq:vHdelBound1}%
\begin{align}
\| \mu^{1/2} \vhHdel \|_T \leq \| \mu^{1/2}\vHdel\|_T &&\forall T\in\Th \text{ (Section \ref{sec:errDec})}, \\
\| \mu^{1/2} \nabla(\phi-\alpha) \|_T \leq \| \mu^{1/2}\vHdel\|_T &&\forall T\in\Th \text{ (Section \ref{sec:errDec})}.
\end{align}
\end{subequations}
In Sections \ref{sec:vthHdelBound} and \ref{sec:tphiBound} we then prove that
\begin{subequations}
\label{eq:vtHdelBound2}
\begin{align}
\|\mu^{1/2} \vthHdel\|_T &\leq C\|\mu^{1/2} \vHdel\|_T &&\forall T\in\Th \text{ (Section \ref{sec:vthHdelBound})}, \label{eq:vthHdelBound} \\
\| \mu^{1/2} \nabla_h(\theta_\nu\tphi-\talpha_\nu)\|_{\omega_\nu} &\leq C\| \mu^{1/2} \vHdel \|_{\omega_\nu} &&\forall \nu\in\Vh\text{ (Section \ref{sec:tphiBound})} \label{eq:tphiBound}.
\end{align}
\end{subequations}
Since 
\begin{align*}
\vtHdel|_T &= \vthHdel|_T + \sum_{\nu:\nu\subset\partial_T} \nabla(\theta_\nu\tphi-\talpha_\nu)|_T,
\end{align*}
we can use the triangle inequality and \eqref{eq:vtHdelBound2} to obtain
\begin{align*}
\|\mu^{1/2}\vtHdel\|_T &\leq \|\mu^{1/2} \vthHdel \|_T + \sum_{\nu:\nu\subset\partial_T} \|\mu^{1/2} \nabla(\theta_\nu\tphi-\talpha_\nu)\|_T \\
&\leq \|\mu^{1/2} \vthHdel \|_T + \sum_{\nu:\nu\subset\partial_T} \|\mu^{1/2} \nabla_h(\theta_\nu\tphi-\talpha_\nu)\|_{\omega_\nu} \\
&\leq C\sum_{T':\overline{T'}\cap\overline{T}\neq\emptyset} \|\mu^{1/2}\vHdel\|_{T'},
\end{align*}
which completes the proof of Theorem \ref{thm:efficiency}. It thus remains to prove \eqref{eq:vHdelBound1} and \eqref{eq:vtHdelBound2}.

\subsection{Decomposition of the error and proof of~\eqref{eq:vHdelBound1}}
\label{sec:errDec}
Define $\vhHdel\in H(\curl;\Th)$ as the unique solution of
\begin{subequations}
\label{eq:vhHdel}
\begin{align}
\nabla\times\vhHdel |_T &= \vjdel_T = \vj|_T-\nabla\times\vH_h|_T, &&  \label{eq:vhHdel1a}\\
(\mu\vhHdel,\nabla\psi)_T &= 0 &&\forall \psi\in H^1(T), \label{eq:vhHdel1b}
\end{align}
\end{subequations}
Since $\nabla\times(\vhHdel-\vthHdel)|_T=\vjdel_T-\vjdel_T \equiv 0$ for each $T\in\Th$, we can define $\phi^\Delta\in H^1(\Th)$ such that
\begin{subequations}
\label{eq:phidel}
\begin{align}
-\nabla_h\phi^\Delta &= \vhHdel-\vthHdel, &&\\
(\phi^\Delta,1)_T &= 0 &&\forall T\in\Th.
\end{align}
\end{subequations}
Now, set $\phi:=\tphi+\phi^\Delta\in H^1(\Th)^3$. We can then write $\vtHdel=\vhHdel+\nabla_h\phi-\nabla\talpha$. Finally, since $\nabla\times(\vHdel-\vtHdel)=\vjdel-\vjdel \equiv \vzero$, we can define $\alpha^\Delta\in H^1(\Omega)$ such that
\begin{subequations}
\label{eq:alpha}
\begin{align}
-\nabla\alpha^\Delta &= \vHdel-\vtHdel, \\
(\alpha^\Delta,1)_{\Omega} &= 0.
\end{align}
\end{subequations}
If we now set $\alpha=\talpha+\alpha^\Delta\in H^1(\Omega)$, we obtain the following decomposition of the error:
\begin{align*}
\vHdel = \vhHdel + \nabla_h\phi - \nabla\alpha.
\end{align*}

Note that, because of (\ref{eq:vhHdel1b}), we have that $(\mu\vhHdel,\nabla(\phi-\alpha))_T=0$ for all $T\in\Th$. From Pythagoras' theorem, it then follows that $\|\mu^{1/2}\vHdel\|^2_T=\|\mu^{1/2}\vhHdel\|_T^2+\|\mu^{1/2}\nabla(\phi-\alpha)\|_T^2$, which proves the bounds in \eqref{eq:vHdelBound1}.

\subsection{Upper bound on $\|\mu^{1/2} \vthHdel\|_T$ in terms of $\|\mu^{1/2} \vHdel\|_T$ (proof of~\eqref{eq:vthHdelBound})}
\label{sec:vthHdelBound}
Firstly, observe that 
\begin{align}
\label{eq:vthHdelMin}
\|\vthHdel\|_T &= \inf_{{\vH' \in \Nd_{\kk}(T),  \nabla\times\vH'=\vjdel_T}} \| \vH'\|_T,
\end{align}
for each $T\in\Th$. Indeed, let $\vH' \in \Nd_{\kk}(T)$ with $\nabla\times\vH'=\vjdel_T$. Then $\nabla\times(\vthHdel|_T-\vH')=\vj_T-\vj_T  \equiv 0$ and so we can write $\vH'-\vthHdel|_T =\nabla \psi$ for some $\psi \in P_{\kk}(T)$. From \eqref{eq:vthHdel1b}, it then follows that $(\vthHdel, \vH'-\vthHdel)_T=(\vthHdel,\nabla \psi)_T=0$ and from Pythagoras' theorem, it then follows that $\|\vH'\|^2_T = \|\vthHdel \|^2_T + \|\vH'-\vthHdel\|_T^2\geq \|\vthHdel\|_T^2$.

In an analogous way, we can show that
\begin{align}
\label{eq:vhHdelMin}
\|\mu^{1/2}\vhHdel\|_T &= \inf_{{\vH' \in H(\curl;T),  \nabla\times\vH'=\vjdel_T}} \|\mu^{1/2}\vH'\|_T.
\end{align}

We also need the following result, which follows from the stability of the regularised Poincar\'e integral operator that was proven in \cite{costabel10}.

\begin{lem}
\label{lem:invCurlBound}
Let $T$ be a tetrahedron. For any $\vr\in\RT_{\kk}(T)\cap H(\dvg^0;T)$, there exists a $\vG\in\Nd_{\kk}(T)$ such that $\nabla\times\vG=\vr$ and
\begin{align*}
\|\vG\|_T &\leq C\inf_{\vG'\in H(\curl;T), \nabla\times\vG'=\vr} \|\vG'\|_T.
\end{align*}
\end{lem}

\begin{proof}
Let $\hat{T}$ denote the reference tetrahedron. We will construct an operator $\vhR:H(\dvg;\hat{T})\rightarrow H(\curl;\hat T)$, independent of $\kk$, such that 
\begin{enumerate}
  \item[C1.] $\nabla\times\vhR\vhr=\vhr$ whenever $\vhr\in H(\dvg^0,\hat T)$.
  \item[C2.] $\vhR\vhr\in\Nd_{\kk}(\hat T)$ whenever $\vhr\in\RT_{\kk}\cap H(\dvg^0;\hat{T})$.
\end{enumerate}
To construct such an operator, define $\vhR_{\vhz}:\cC^{\infty}(\overline{\hat{T}})\rightarrow H(\curl;\hat T)$, for any $\vhz\in\hat{T}$, as the following Poincar\'e integral operator:
\begin{align*}
\vhR_{\vhz}\vhr(\vhx) &:= -(\vhx-\vhz)\times \int_0^1 \tau\vhr(\tau(\vhx-\vhz)+\vhz)\;\mathrm{d}\tau.
\end{align*}
This operator is based on the integral operator used in \cite[Theorem 4.11]{spivak65}. The operator $\vhR_{\vhz}$ can be extended to $H(\dvg;\hat T)$ and satisfies conditions C1 and C2 \cite[Theorem 2.1 and Remark 3.3]{gopalakrishnan04}. Now, let $B$ be an open ball in $\hat{T}$ and let $\vartheta\in\cC_0^\infty(\hat T)$ be an analytic function with support on $B$ such that $\int_{\hat T} \vartheta(\vhx)\;\dhx = 1$. We then define $\vhR:\cC^{\infty}(\overline{\hat{T}})\rightarrow H(\curl;\hat T)$ as the following regularised Poincar\'e integral operator:
\begin{align*}
\vhR\vhr(\vhx):= \int_{\hat T} \vartheta(\vhz)\vhR_{\vhz}\vhr(\vhx) \;\dhz.
\end{align*}
Since $\vhR_{\vhz}$, for every $\vhz\in\hat T$, can be extended to $H(\dvg;\hat T)$ and satisfies conditions C1 and C2, so does $\vhR$. By applying the coordinate transformations $\vhy=\tau(\vhx-\vhz)+\vhz$ and $t=(1-\tau)^{-1}$, the above can be rewritten as
\begin{align*}
\vhR\vhr(\vhx) = \int_{\hat T} -(\vhx-\vhy)\times\vhr(\vhy) \left(\int_{1}^\infty t(t-1)\vartheta(\vhx+t(\vhy-\vhx)) \;\mathrm{d}t\right) \;\dhy,
\end{align*}
where $\vartheta$ and $\vhr$ are extended by zero to $\mathbb{R}^3$. This is exactly the operator $R_2$ of \cite[Definition 3.1]{costabel10}. By taking $s=-1$ in \cite[Corollary 3.4]{costabel10}, it follows that
\begin{align}
\label{eq:invCurlBound1}
\| \vhR\vhr \|_{\hat T} &\leq C \|\vhr\|_{H^{-1}(\hat T)} &&\forall \vhr\in H(\dvg^0,\hat T),
\end{align}
where we stress once more that the operator $\vhR$ and the constant $C$ are independent of $\kk$.

Now, let $\vhG'\in H(\curl;\hat T)$ and $\vhr\in H(\dvg^0,\hat T)$ be two functions such that $\nabla\times\vhG'=\vhr$. Then
\begin{align*}
\|\vhr\|_{H^{-1}(\hat T)^3} &= \sup_{\vw\in H_0^1(\hat T)^3\setminus\{0\}} \frac{(\vhr,\vw)_{\hat T}}{\|\vw\|_{H^1(\hat T)^3}} \\
&= \sup_{\vw\in H_0^1(\hat T)^3 \setminus\{0\}} \frac{(\nabla\times\vhG',\vw)_{\hat T}}{\|\vw\|_{H^1(\hat T)^3}} \\
&= \sup_{\vw\in H_0^1(\hat T)^3 \setminus\{0\}} \frac{(\vhG',\nabla\times\vw)_{\hat T}}{\|\vw\|_{H^1(\hat T)^3}} \\
&\leq \sup_{\vw\in H_0^1(\hat T)^3 \setminus\{0\}} \frac{\|\vhG'\|_{\hat T} \|\nabla\times\vw\|_{\hat T}}{\|\vw\|_{H^1(\hat T)^3}} \\
&\leq \sqrt{2} \|\vhG'\|_{\hat T}
\end{align*}
where $\|\vw\|_{H^1(\hat T)^3}^2:= \|\vw\|_{\hat T}^2+\|\nabla\vw\|_{\hat T}^2$ and where the fourth line follows from the Cauchy--Schwarz inequality and the last line from the fact that $\|\nabla\times\vw\|_{\hat T}\leq \sqrt{2} \|\vw\|_{H^1(\hat T)^3}$. From \eqref{eq:invCurlBound1}, it then follows that
\begin{align}
\label{eq:invCurlBound2}
\|\vhR\vhr\|_{\hat T} &\leq C \inf_{\vhG'\in H(\curl;\hat T), \nabla\times\vhG'=\vhr} \|\vhG'\|_{\hat T}.
\end{align}

Now, let $\boldsymbol{\varphi}_T:\hat{T}\rightarrow T$ denote the affine element mapping and let $\jac_T := [\frac{\partial \boldsymbol{\varphi}_T}{\partial\hat{x}_1}\, \frac{\partial \boldsymbol{\varphi}_T}{\partial\hat{x}_2}\, \frac{\partial \boldsymbol{\varphi}_T}{\partial\hat{x}_3}]$ be the Jacobian of $\boldsymbol{\varphi}_T$, with $\frac{\partial \boldsymbol{\varphi}_T}{\partial\hat{x}_i}$ column vectors. We define the covariant transformation $\vT_{T,\curl}:H(\curl;\hat T)\rightarrow H(\curl;T)$ and the Piola contravariant transformation $\vT_{T,\dvg}:H(\dvg;\hat T)\rightarrow H(\dvg;T)$ such that
\begin{align*}
\vT_{T,\curl}\vhG\circ \boldsymbol{\varphi}_T &:=\jac_T^{-t} \vhG, \qquad \vT_{T,\dvg}\vhr\circ \boldsymbol{\varphi}_T := \frac{1}{\mathrm{det}(\jac_T)}\jac_T \vhr,
\end{align*}
where $\jac_T^{-t}$ denotes the transposed of the inverse of $\jac_T$ and $\mathrm{det}(\jac_T)$ denotes the determinant of $\jac_T$. We set $\vG=\vT_{T,\curl}\vhR\vT_{T,\dvg}^{-1}\vr$. Then $\nabla\times\vG=\vr$. For any $\vG'\in H(\curl;T)$ that satisfies $\nabla\times\vhG'=\vr$, we can then derive
\begin{align*}
\|\vG\|_T &\leq Ch_T^{3/2} \|\vT_{T,\curl}^{-1}\vG\|_{\hat T} \\
&= Ch_T^{3/2} \|\vhR\vT_{T,\dvg}^{-1}\vr \|_{\hat T} \\
&\leq Ch_T^{3/2} \| \vT_{T,\curl}^{-1}\vG' \|_{\hat T} \\
&\leq C \| \vG' \|_{T},
\end{align*}
where $h_T$ denotes the diameter of $T$, where the first and last lines follow from standard scaling arguments, and where the third line follows from \eqref{eq:invCurlBound2} and the fact that $\nabla\times \vT_{T,\curl}^{-1}\vG' = \vT_{T,\dvg}^{-1}\nabla\times\vG' = \vT_{T,\dvg}^{-1}\vr$. This then proves the lemma.
\end{proof}

From \eqref{eq:vthHdelMin}, Lemma \ref{lem:invCurlBound}, and \eqref{eq:vHdelBound1}, it follows that
\begin{align*}
\|\mu^{1/2} \vthHdel\|_T \leq C \| \mu^{1/2} \vhHdel\|_T \leq C \| \mu^{1/2} \vHdel \|_T
\end{align*}
for all $T\in\Th$, which proves \eqref{eq:vthHdelBound}.

\subsection{Upper bound on $\| \mu^{1/2} \nabla_h(\theta_\nu\tphi-\talpha_\nu)\|_{\omega_\nu}$ in terms of $\| \mu^{1/2} \vHdel \|_{\omega_\nu}$ (proof of~\eqref{eq:tphiBound})}
\label{sec:tphiBound}
For all $\nu\in\Vh$, define $\Fnu^I:=\{f\in\Fhin \;|\; \partial f\ni \nu\}$ as the set of all internal faces that are connected to $\nu$. Observe that 
\begin{align}
\label{eq:talpha_nu2}
\|\mu^{1/2} \nabla_h(\theta_\nu\tphi-\talpha_\nu)\|_{\omega_\nu} = \inf_{\substack{u'\in P_{\kk+1}^{-1}(\Tnu), \\ \ju{u'}_f=\ju{\theta_\nu\tphi}_f \;\forall f\in\Fnu^I, \\ u'|_f=0 \;\forall f\subset\Gamma_\nu}} \| \mu^{1/2}  \nabla_h u'\|_{\omega_\nu}
\end{align}
for all $\nu\in\Vh$. Indeed, let $u'\in P^{-1}_{\kk+1}(\Tnu)$ such that $\ju{u'}_f=\ju{\theta_\nu\tphi}_f$ for all $f\in\Fnu^I$ and $u'|_f=0$ for all $f\subset{\Gamma_\nu}$. Then $u'-\theta_\nu\tphi \in P_{\kk+1,0,\Gamma_\nu}(\omega_\nu)$ and so $w:=u'- (\theta_\nu\tphi -\talpha_\nu) \in P_{\kk+1,0,\Gamma_\nu}(\omega_\nu)$. Using \eqref{eq:talpha_nu}, we can then derive
\begin{align*}
\big(\mu\nabla_h(\theta_\nu\tphi -\talpha_\nu),\nabla_hw \big)_{\omega_\nu} = \big(\mu\nabla_h(\theta_\nu\tphi)-\mu\nabla\talpha_\nu,\nabla w\big)_{\omega_\nu} = 0.
\end{align*}
From Pythagoras' theorem it then follows that 
\begin{align*}
\|\mu^{1/2} \nabla_hu'\|^2_{\omega_\nu} &= \|\mu^{1/2} \nabla_h(\theta_\nu\tphi -\talpha_\nu) \|^2_{\omega_\nu} + \|\mu^{1/2} \nabla w\|^2_{\omega_\nu} \geq \| \mu^{1/2}  \nabla_h(\theta_\nu\tphi -\talpha_\nu) \|^2_{\omega_\nu},
\end{align*}
which proves \eqref{eq:talpha_nu2}.

Now, define $\alpha_\nu\in H^1_{0,\Gamma_\nu}(\omega_\nu)$ such that
\begin{align*}
(\mu\nabla\alpha_\nu,\nabla w)_{\omega_\nu} &= (\mu\nabla_h(\theta_\nu\tphi),\nabla w)_{\omega_\nu} &&\forall w\in H^1_{0,\Gamma_\nu}(\omega_\nu).
\end{align*}
In a similar way as for the discrete case \eqref{eq:talpha_nu2}, one can prove that
\begin{align}
\label{eq:alpha_nu}
\|\mu^{1/2} \nabla_h(\theta_\nu\tphi-\alpha_\nu)\|_{\omega_\nu} = \inf_{\substack{u'\in H^1(\Tnu), \\ \ju{u'}_f=\ju{\theta_\nu\tphi}_f \;\forall f\in\Fnu^I, \\ u'|_f=0 \;\forall f\subset\Gamma_\nu}} \|\mu^{1/2}  \nabla_h u'\|_{\omega_\nu}
\end{align}
for all $\nu\in\Vh$.

From \cite[Theorem 2.4]{ern19}, it follows that
\begin{align*}
\inf_{\substack{u'\in P_{\kk+1}^{-1}(\Tnu), \\ \ju{u'}_f=\ju{\theta_\nu\tphi}_f \;\forall f\in\Fnu^I, \\ u'|_f=0 \;\forall f\subset\Gamma_\nu}} \|\nabla_h u'\|_{\omega_\nu} &\leq C\inf_{\substack{u'\in H^1(\Tnu), \\ \ju{u'}_f=\ju{\theta_\nu\tphi}_f \;\forall f\in\Fnu^I, \\ u'|_f=0 \;\forall f\subset\Gamma_\nu}} \|\nabla_h u'\|_{\omega_\nu}
\end{align*}
for all $\nu\in\Vh$. From \eqref{eq:talpha_nu2}, \eqref{eq:alpha_nu}, and the above, it then follows that
\begin{align}
\label{eq:talphaBound}
\|\mu^{1/2}  \nabla_h(\theta_\nu\tphi-\talpha_\nu)\|_{\omega_\nu} &\leq C \| \mu^{1/2}  \nabla_h(\theta_\nu\tphi-\alpha_\nu)\|_{\omega_\nu}
\end{align}
for all $\nu\in\Vh$. Properties \eqref{eq:talpha_nu2}, \eqref{eq:alpha_nu}, and \eqref{eq:talphaBound} are also a consequence of \cite[Corollary 3.1, Remark 3.2]{ern19}.

It now remains to derive an upper bound on
$\|\nabla_h(\theta_\nu\tphi-\alpha_\nu)\|_{\omega_\nu}$ in terms of
$\|\vHdel\|_{\omega_\nu}$. To do this, we need the following result,
which follows immediately from \cite[Theorem 5.1, Remark
5.3]{brenner03}; for completeness, we report a proof of it in
Appendix~\ref{appA}. 

\begin{prp}
\label{prp:phiProjErr}
For every $u\in H^1(\Tnu)$, with $(\ju{u},1)_f=0$ for each $f\in\Fnu^I$, we have that
\begin{align*}
\|u-\overline{u}^{\omega_\nu} \|_{\omega_\nu} \leq Ch_\nu \| \nabla_hu\|_{\omega_\nu},
\end{align*}
where $\overline{u}^{\omega_\nu}$ denotes the average of $u$ in $\omega_\nu$. 
\end{prp}

Now, note that $\ju{\theta_\nu\tphi-\alpha_\nu}_f =  \ju{\theta_\nu\tphi}_f = \ju{\theta_\nu(\tphi - \alpha - \overline{(\tphi - \alpha )}^{\omega_\nu})}_f$ for all $f\in\Fnu^I$. Using \eqref{eq:alpha_nu}, we can then derive
\begin{align*}
&\|\mu^{1/2}\nabla_h(\theta_\nu\tphi-\alpha_\nu)\|_{\omega_\nu} \leq \| \mu^{1/2}\nabla_h (\theta_\nu(\tphi - \alpha - \overline{(\tphi - \alpha )}^{\omega_\nu})) \|_{\omega_\nu} \\
&\qquad= \|\mu^{1/2}(\nabla\theta_\nu) (\tphi - \alpha - \overline{(\tphi - \alpha )}^{\omega_\nu}) + \mu^{1/2}\theta_\nu \nabla_h(\tphi - \alpha - \overline{(\tphi - \alpha )}^{\omega_\nu}) \|_{\omega_\nu} \\
&\qquad\leq \| \mu^{1/2}(\nabla\theta_\nu) (\tphi - \alpha - \overline{(\tphi - \alpha )}^{\omega_\nu}) \|_{\omega_\nu} + \| \mu^{1/2}\theta_\nu \nabla_h(\tphi - \alpha - \overline{(\tphi - \alpha )}^{\omega_\nu}) \|_{\omega_\nu} \\
&\qquad\leq Ch_{\nu}^{-1} \| \tphi - \alpha - \overline{(\tphi - \alpha )}^{\omega_\nu} \|_{\omega_\nu} + \| \nabla_h(\tphi - \alpha - \overline{(\tphi - \alpha )}^{\omega_\nu}) \|_{\omega_\nu} \\
&\qquad\leq C\| \nabla_h(\tphi - \alpha) \|_{\omega_\nu} \\
&\qquad\leq C\| \mu^{1/2} \nabla_h(\tphi - \alpha) \|_{\omega_\nu}
\end{align*}
for all $\nu\in\Vh$, where the fifth line follows from Proposition~\ref{prp:phiProjErr}, \eqref{eq:tlambda1b}, and \eqref{eq:tphi1a}. Now, recall that $\tphi-\alpha=\phi-\alpha - \phi^\Delta$ and $\nabla_h\phi^\Delta=\vthHdel-\vhHdel$ (see Section \ref{sec:errDec}). We can use the triangle inequality, \eqref{eq:vthHdelBound}, and \eqref{eq:vHdelBound1} to derive
\begin{align*}
\| \mu^{1/2} \nabla_h(\tphi-\alpha) \|_{\omega_\nu} &\leq \|\mu^{1/2} \nabla_h(\phi-\alpha)\|_{\omega_\nu} + \| \mu^{1/2} \nabla_h\phi^\Delta\|_{\omega_\nu} \\
&\leq \|\mu^{1/2} \nabla_h(\phi-\alpha)\|_{\omega_\nu} + \|\mu^{1/2} \vhHdel\|_{\omega_\nu}  + \|\mu^{1/2}\vthHdel \|_{\omega_\nu} \\
&\leq C\| \mu^{1/2} \vHdel \|_{\omega_\nu}
\end{align*}
for all $\nu\in\Vh$. Inequality \eqref{eq:tphiBound} then follows from the last two inequalities and~\eqref{eq:talphaBound}.

\section{Numerical experiments}\label{sec:numerics}
In the following, we investigate the reliability, efficiency, and polynomial-degree robustness
of the equilibrated {\em a posteriori} error estimator $\eta_h$ constructed following Steps 1-5 in Section \ref{sec:errEst}.
We present numerical experiments for the unit cube and the
L-brick domain on the same test problems as in our previous work~\cite{gedicke20}.
In all experiments we 
set $\mu=1$ unless stated otherwise.
As in~\cite{gedicke20}, we do not project the right hand side $\mathbf{j}$ onto $D_k(\mathcal{T}_h)\cap H(\dvg^0;\Omega)$.
This introduces small compatibility errors in Steps 1-3 that can be neglected.
We investigate the reliability and efficiency of $\eta_h$ for
uniformly refined and adaptively refined
meshes. For the adaptive mesh refinement, we employ the standard adaptive finite element loop, \emph{solve}, \emph{estimate}, \emph{mark}, and \emph{refine}. 
We use a multigrid preconditioned conjugate gradient solver \cite{H1999}, choose $\theta=0.5$ in the bulk marking strategy \cite{Doerfler1996}, and refine the mesh using a bisection strategy \cite{AMP2000}. In order to ensure that the discretisation of $\vj$ is compatible, we add a small gradient correction term following \cite[Section 4.1]{creuse19a}.

\subsection{Unit cube examples}
\begin{figure}[t]
\includegraphics[width=0.49\textwidth]{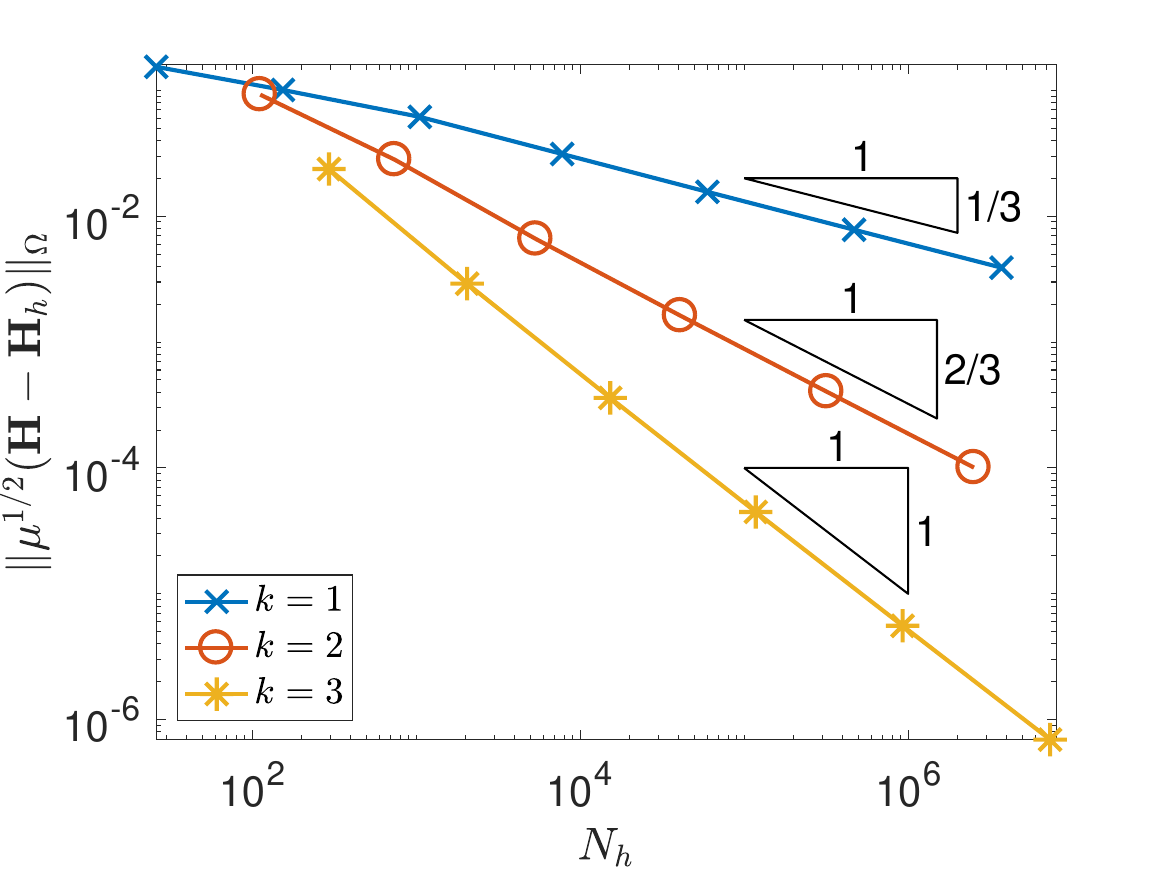}
\includegraphics[width=0.49\textwidth]{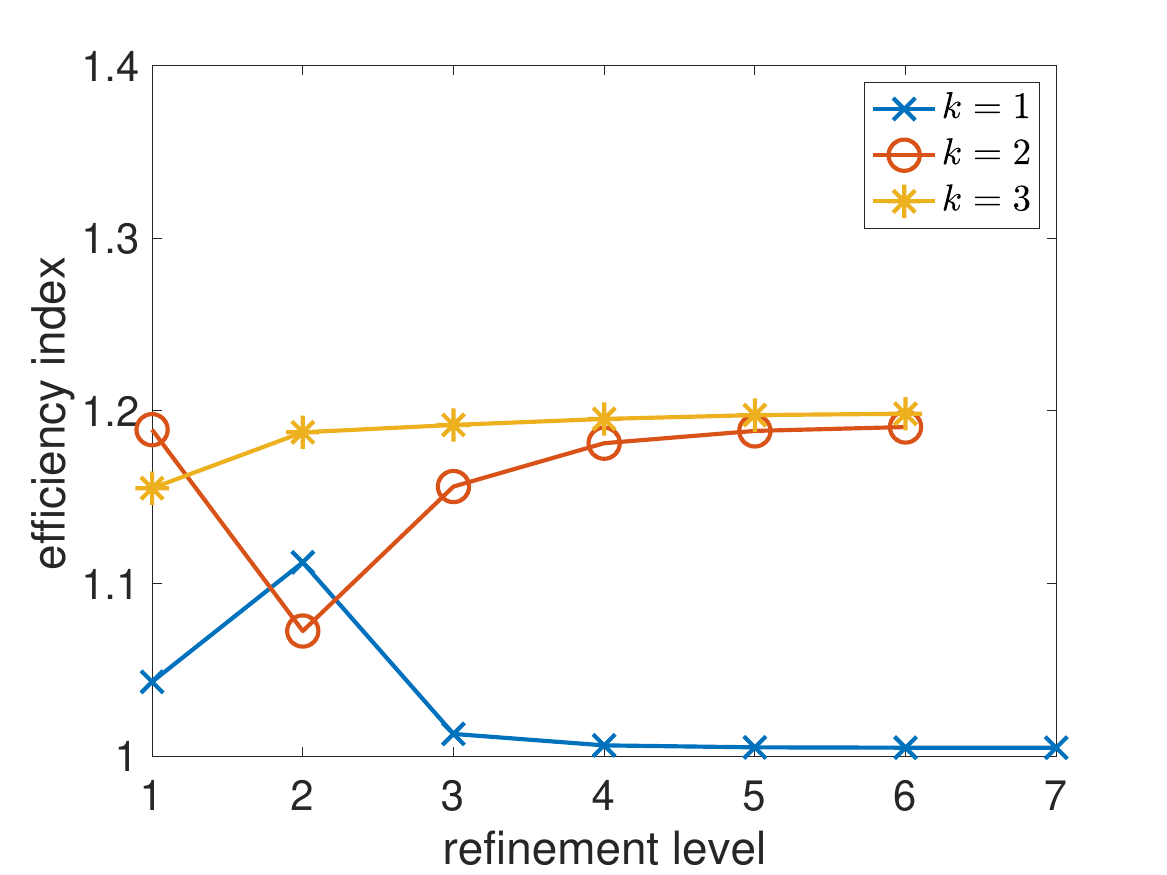}
\caption{Error and efficiency indices for the unit cube example with
  polynomial solution and
 uniformly refined
  meshes.}
\label{fig:unit:cube}
\end{figure}
\begin{figure}[t]
\includegraphics[width=0.49\textwidth]{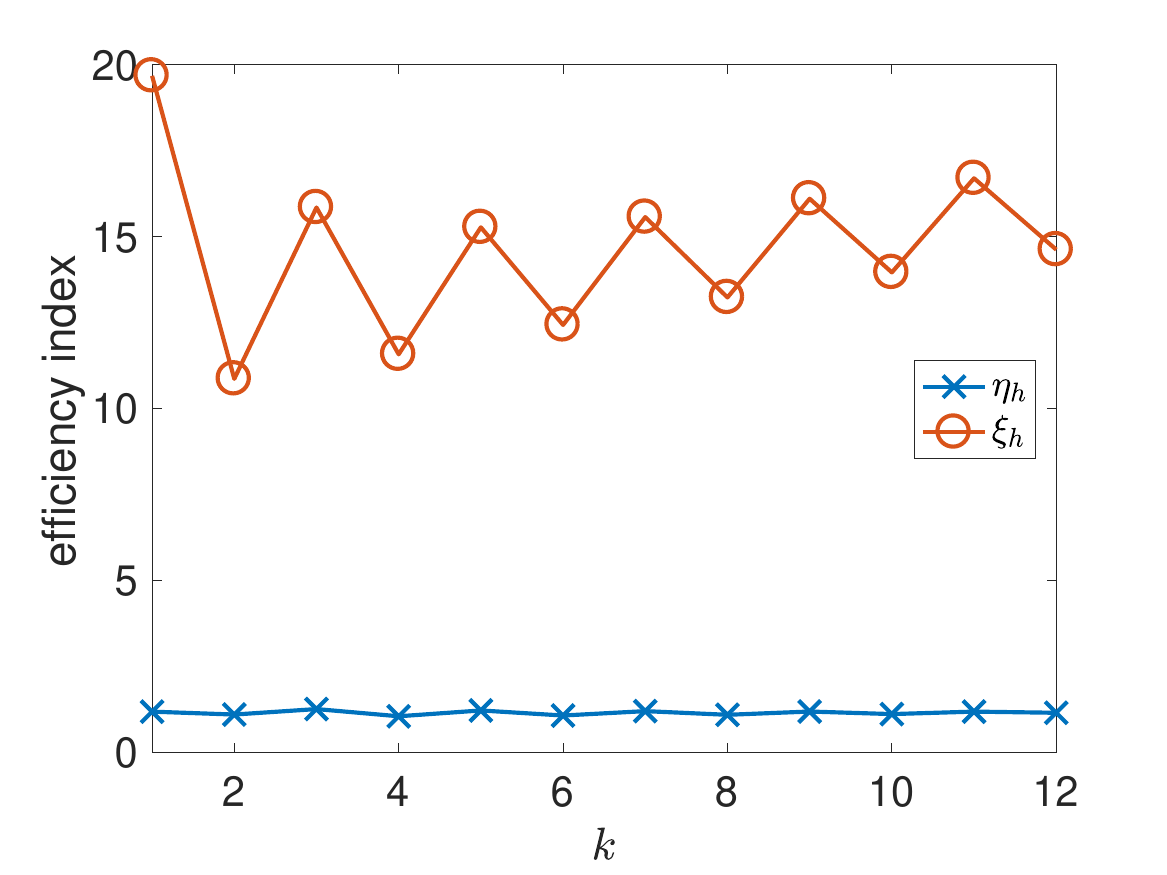}
\includegraphics[width=0.49\textwidth]{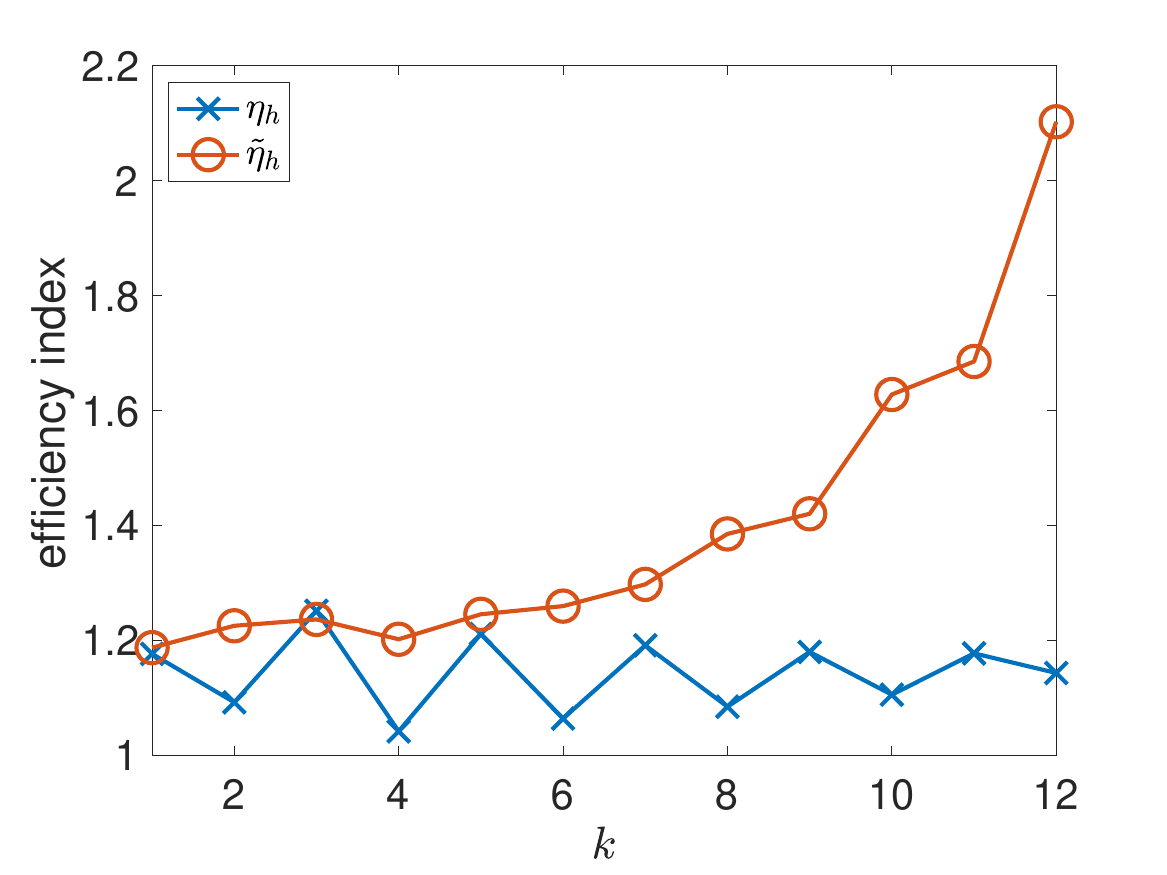}
\caption{Polynomial robustness of the equilibrated {\em a posteriori} error estimator $\eta_h$ in comparison to the residual {\em a posteriori} error estimator $\xi_h$ (left) and the equilibrated estimator $\tilde\eta_h$ (right)
for the second unit cube example and a quasi-uniform mesh with
24 elements.}
\label{fig:unit:cube:2}
\end{figure}
In this example, we solve the 
Maxwell problem on the unit cube $\Omega=(0,1)^3$ with
$\vn\cdot\vH = 0$ on $\partial \Omega$,
for two different right hand sides.
\par

Firstly, we choose the right-hand side 
$\mathbf{j}$ according to
the polynomial solution 
\[
\vH=\nabla\times\vu, \qquad \mathbf{u}(x,y,z) = \left( \begin{array}{c} y(1-y)z(1-z)\\ x(1-x)z(1-z) \\x(1-x)y(1-y) \end{array}\right).
\]
The errors $\| \mu^{1/2}(\vH-\vH_h) \|_{\Omega}$ and efficiency indices
$\eta_h/\| |\mu^{1/2}(\vH-\vH_h) \|_{\Omega}$ are presented in
Figure~\ref{fig:unit:cube} for $k=1,2,3$ and uniformly refined
meshes.
We observe optimal rates $\mathcal{O}(h^k) = \mathcal{O}(N_h^{-k/3})$, 
$N_h = \operatorname{dim}(R_k(\mathcal{T}_h))$, for the convergence of the errors, 
and efficiency indices between $1$ and $2$.
Note that, for $k=3$, $\mathbf{j}\in D_k(\mathcal{T}_h)\cap H(\dvg^0;\Omega)$, hence in that case there is no compatibility error.
\par

For the investigation of the robustness with respect to the polynomial degree $k$, we consider the right-hand side
$\mathbf{j}$ according to the non-polynomial solution 
\[
\vH=\nabla\times\vu, \qquad 
\mathbf{u}(x,y,z) = \left( \begin{array}{c} 
\sin(\pi y)\sin(\pi z)\\  
\sin(\pi x)\sin(\pi z)\\
\sin(\pi x)\sin(\pi y) 
\end{array}\right).
\]
We compare the efficiency indices for $k$-refinement of $\eta_h$ to those of the
residual {\em a posteriori} error estimator \cite{beck00}
\[ 
\xi_h^2 := \sum_{T\in\Th}  \frac{ \mu_T h_T^2}{k^2} \| \vj - \nabla\times\vH_h \|_{T}^2
+ \sum_{f\in\Fhin}\frac{ \mu_f h_f}{k}  \| \jut{ \vH_h } \|_{f}^2,
\]
where $\mu_T$ is the value of $\mu$ at element $T$ and $\mu_f$ is the average value of $\mu_T$ of the elements $T$ adjacent to $f$.
We also compare the efficiency indices to those of the equilibrated
{\em a posteriori} error estimator
\[ 
\tilde\eta_h := \| \mu^{1/2} (\hat{\vtH}^{\Delta} +\nabla_h\tilde\phi) \|_{\Omega}
\]
of our previous work \cite{gedicke20}, which does not include the computation of $\tilde\alpha$. 
We observe in Figure~\ref{fig:unit:cube:2} that both the efficiency indices
for $\xi_h$ and $\tilde\eta_h$ grow in $k$ (although $\tilde\eta_h$ remains confined to small values, for all tested polynomial degrees), while those of $\eta_h$ are stable in $k$.

\subsection{L-brick example}
\begin{figure}[t]
\includegraphics[width=0.49\textwidth]{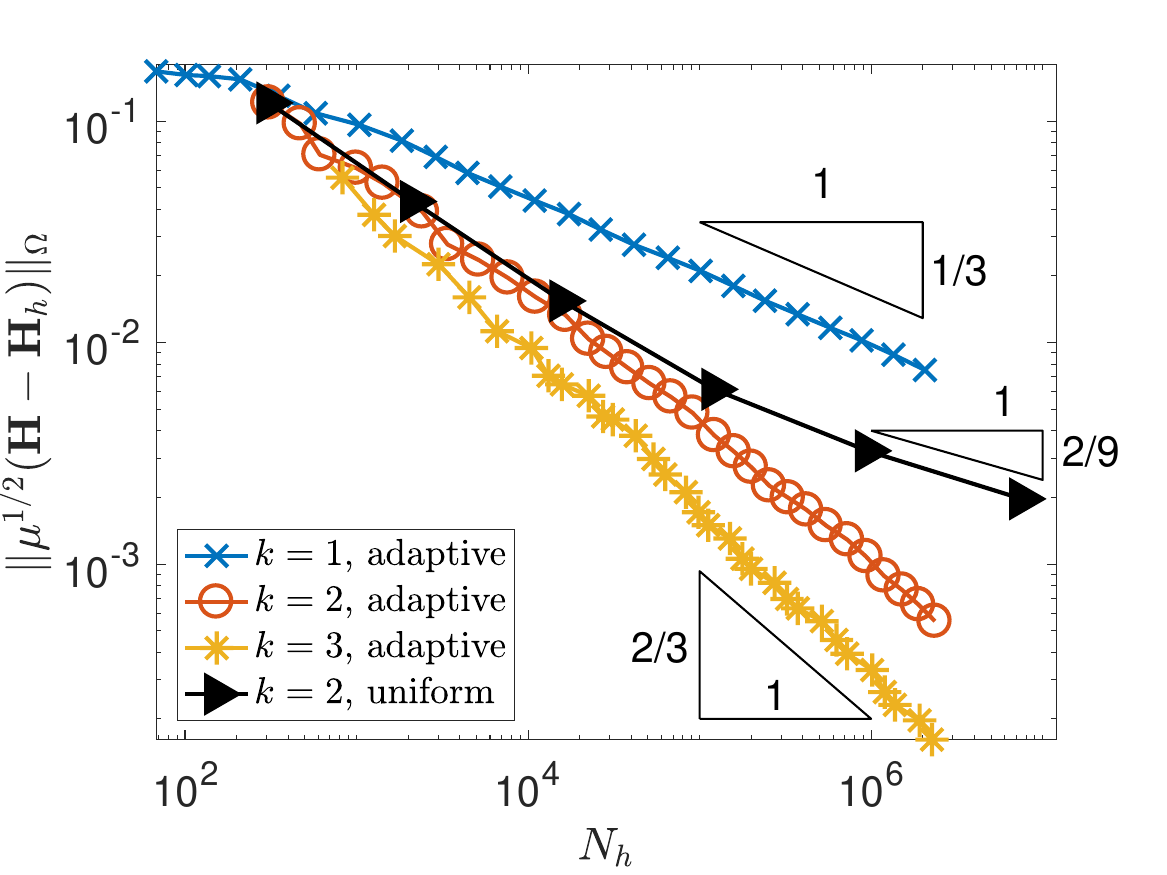}
\includegraphics[width=0.49\textwidth]{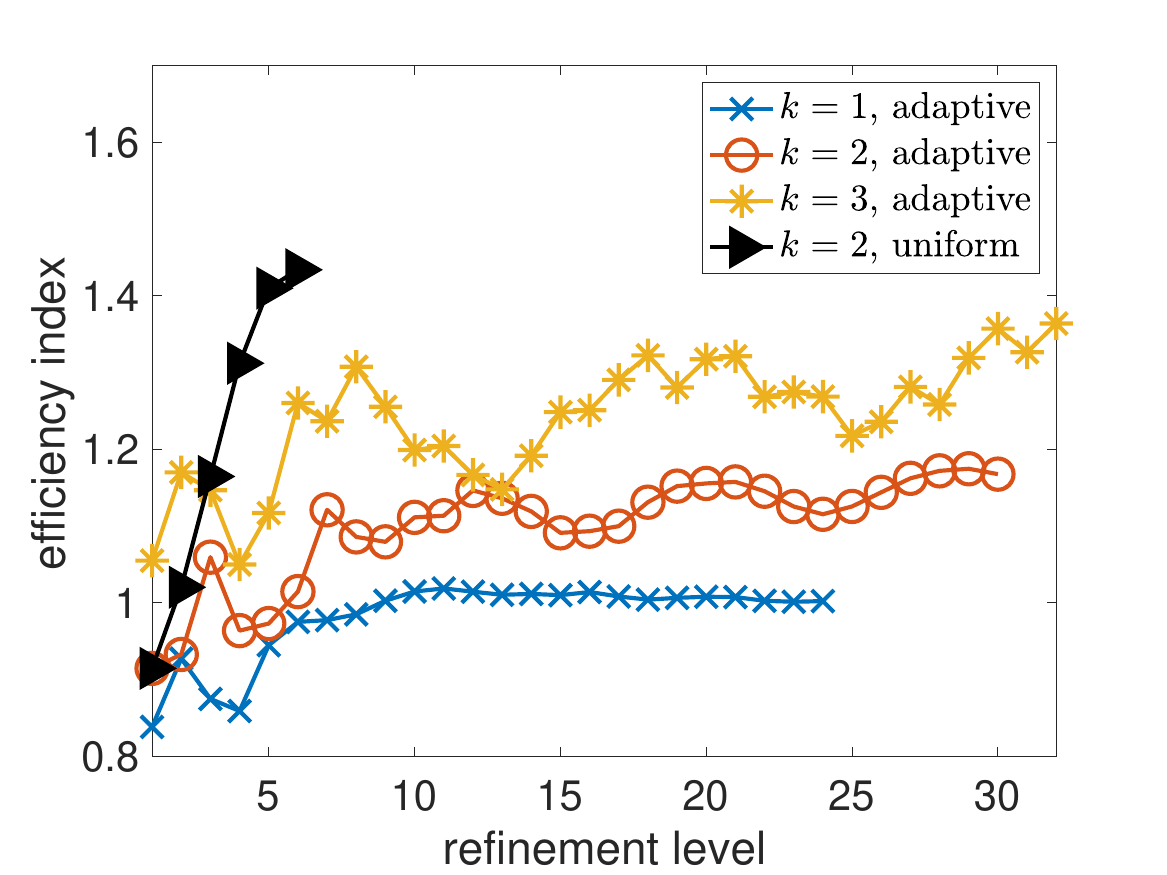}
\caption{Error and efficiency indices for adaptive mesh refinement for the L-brick example.}
\label{fig:L-brick}
\end{figure}
\begin{figure}[t]
\includegraphics[width=0.49\textwidth]{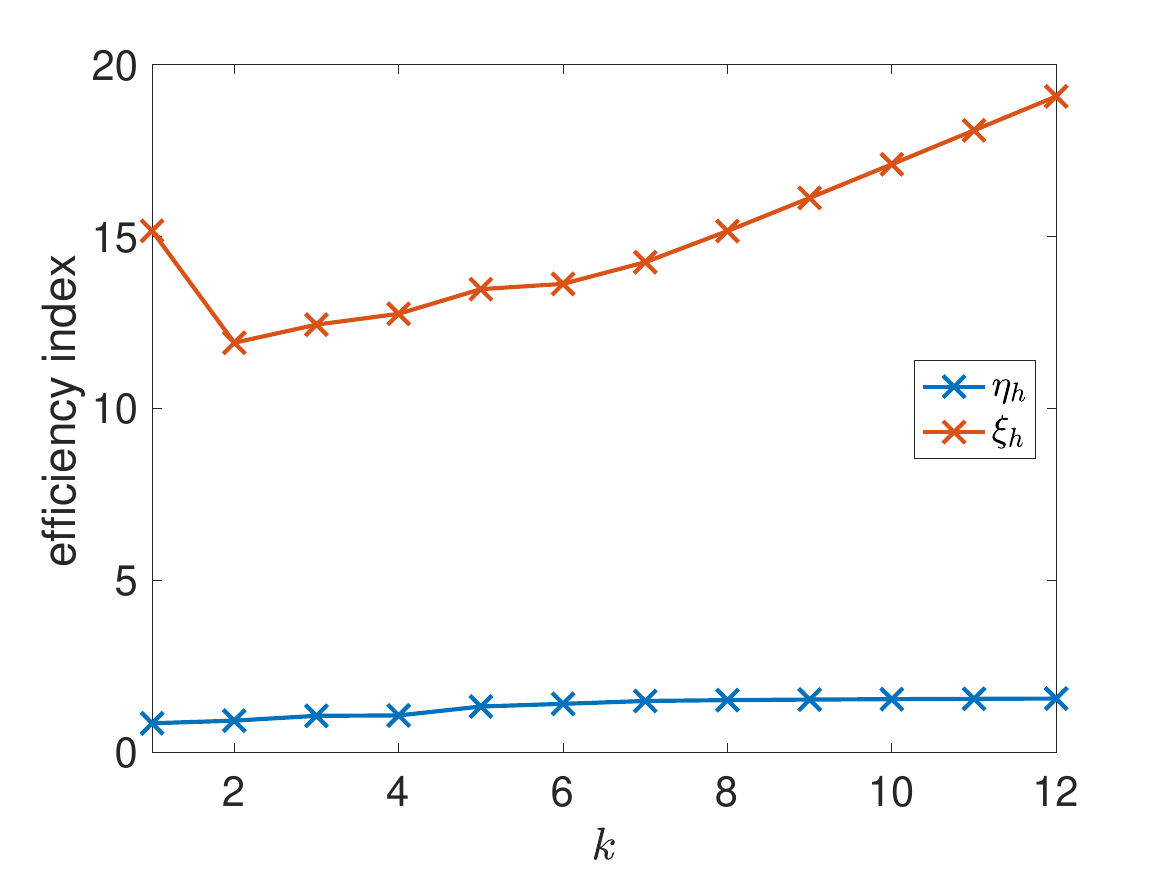}
\includegraphics[width=0.49\textwidth]{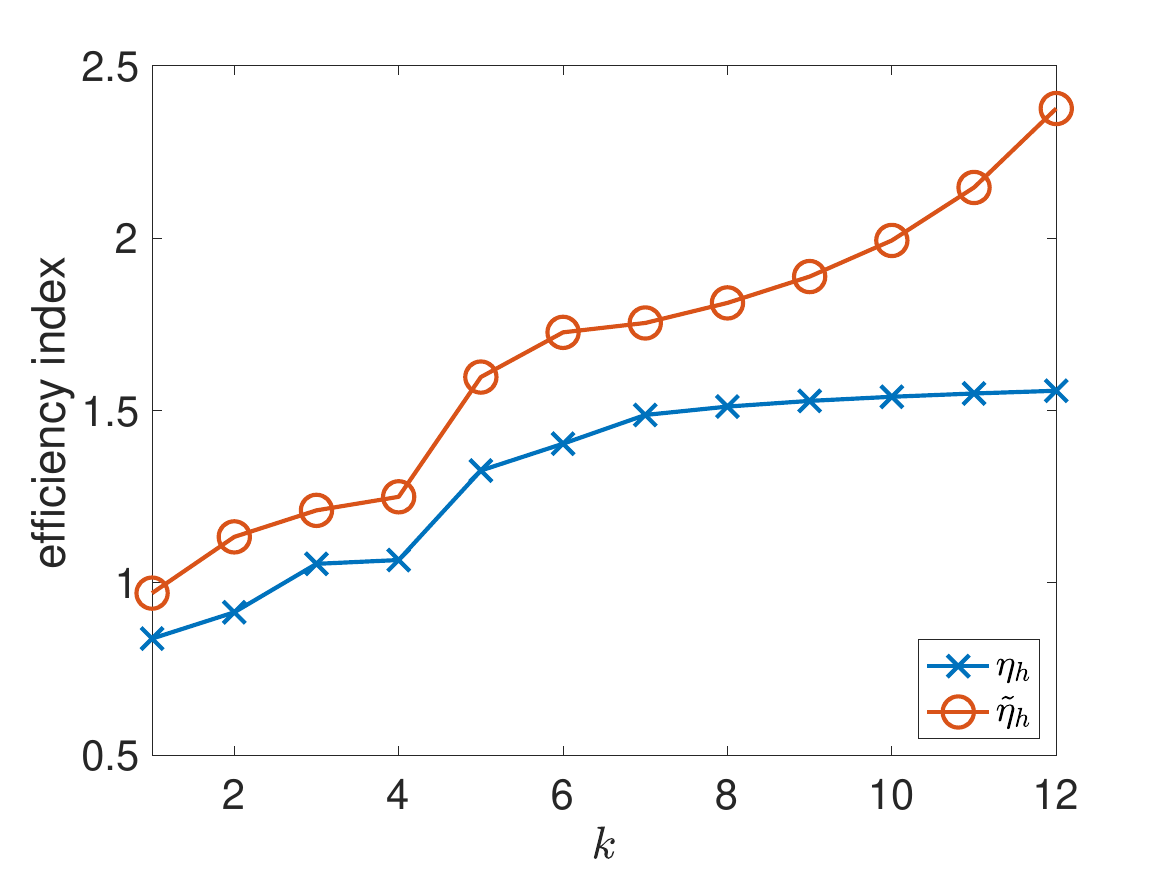}
\caption{Polynomial robustness of the equilibrated {\em a posteriori} error estimator $\eta_h$ in comparison to the residual {\em a posteriori} error estimator $\xi_h$ (left) and the equilibrated estimator $\tilde\eta_h$ (right)
for the L-brick example and a quasi-uniform mesh with 36 elements.}
\label{fig:L-brick:robustness}
\end{figure}
In this example, we consider the homogeneous Maxwell problem
on the (nonconvex) domain
\[
  \Omega = (-1,1)\times(-1,1)\times(0,1)\backslash \left(
    [0,1]\times[-1,0]\times[0,1]\right).
  \]
We choose the right-hand side $\mathbf{j}$ according to the singular solution
\[
\vH=\nabla\times\vu, \quad
\mathbf{u}(x,y,z) = \nabla \times \left( \begin{array}{c} 0\\ 0 \\(1-x^2)^2(1-y^2)^2((1-z)z)^2r^{2/3}\cos(\frac{2}{3}\varphi) \end{array}\right),
\]
where $(r,\varphi)$ are the two dimensional polar coordinates in the $x$-$y$-plane.

In Figure~\ref{fig:L-brick}, we observe suboptimal convergence rates
of asymptotically $\mathcal{O}(N_h^{-2/9})$
for uniform mesh refinement and $k=2$, due to the edge singularity.
For adaptive mesh refinement, we observe improved convergence
rates of $\mathcal{O}({N_h}^{-1/3})$ for $k=1$, close to $\mathcal{O}(({N_h}/\ln({N_h}))^{-2/3})$ for $k=2$, and 
of $\mathcal{O}({N_h}^{-2/3})$ for $k\geq3$, which are in fact the best possible rates one can get with
isotropic mesh refinement, cf. \cite[section 4.2.3]{A1999}.
Again, we observe efficiency indices between 1 and 2.

To investigate the $k$-robustness of the estimator $\eta_h$, 
we compare the efficiency indices for $k$-refinement
of $\eta_h$ to those of $\xi_h$ and $\tilde\eta_h$.
In Figure~\ref{fig:L-brick:robustness}, we observe the same
as for the unit cube example, namely, that the new estimator $\eta_h$
is robust with respect to the polynomial degree $k$, while the
residual estimator $\xi_h$, as well as the equilibrated estimator
$\tilde\eta_h$, is 
not robust in $k$.

\subsection{Example with discontinuous permeability}
\begin{figure}[t]
\includegraphics[width=0.49\textwidth]{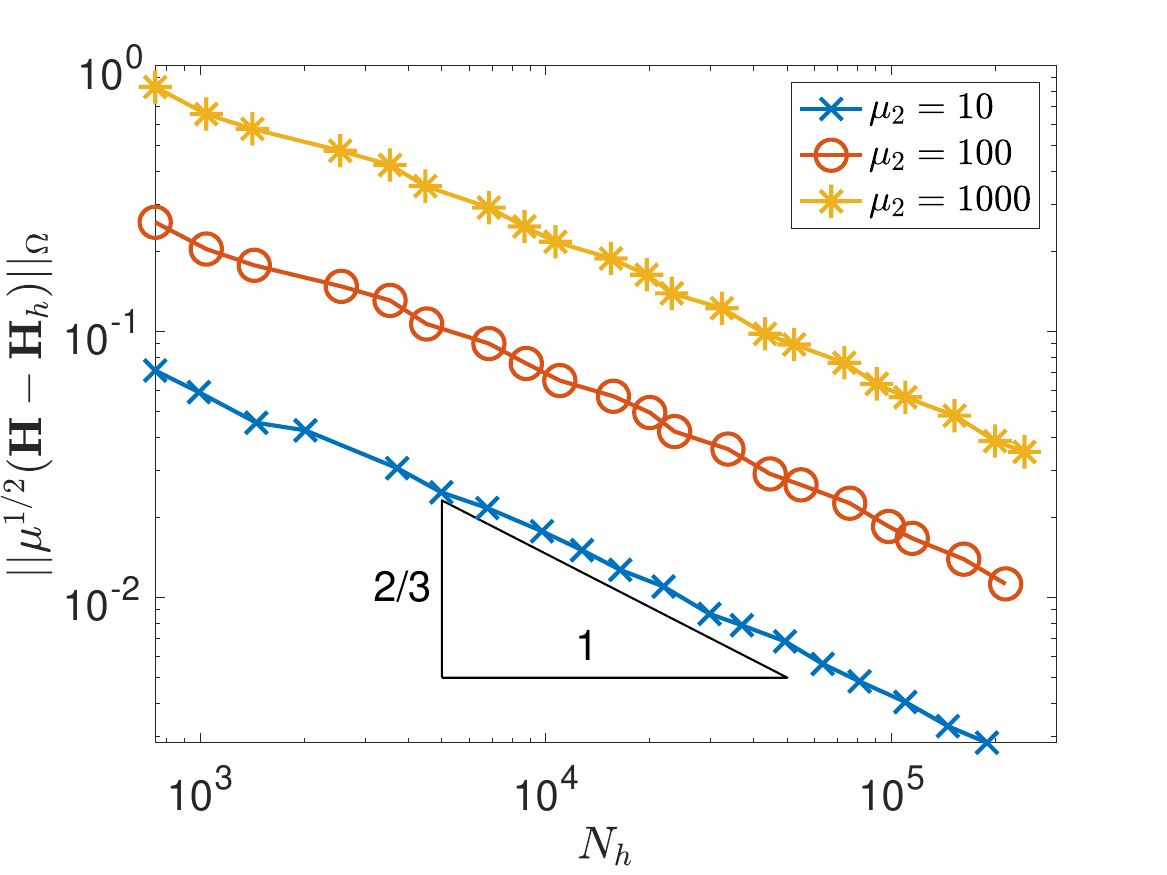}
\includegraphics[width=0.49\textwidth]{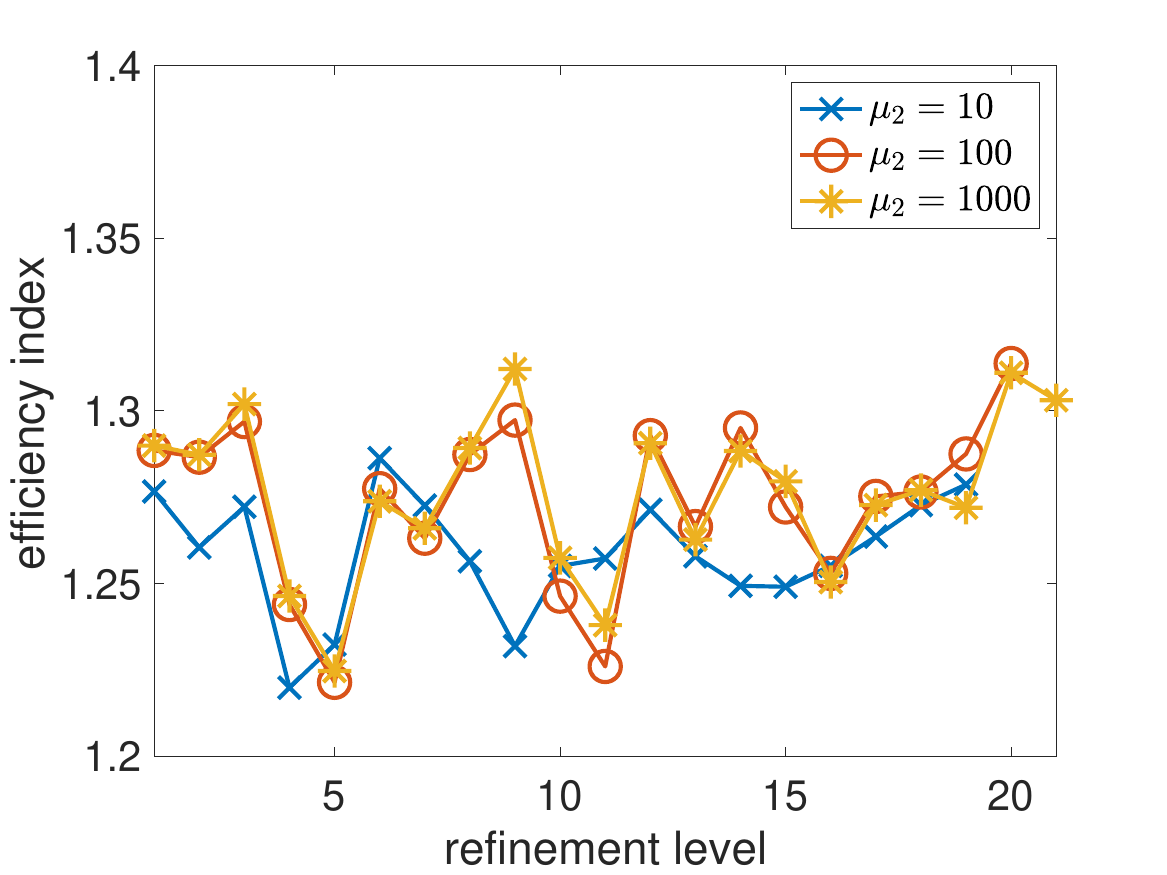}
\caption{Error and efficiency indices for adaptive mesh refinement for the example with discontinuous permeability.}
\label{fig:VarCoeff}
\end{figure}
\begin{figure}[t]
\includegraphics[width=0.49\textwidth]{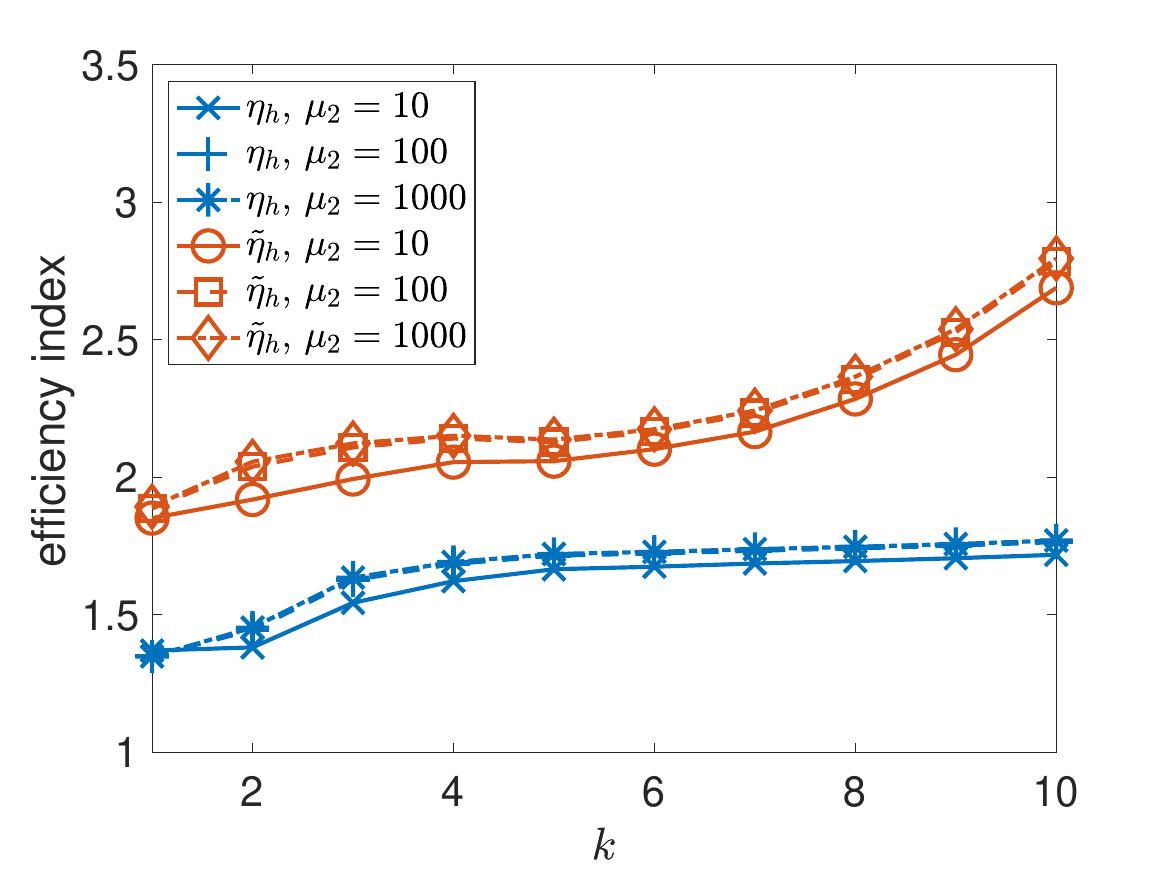}
\caption{Polynomial robustness of the {\em a posteriori} error estimator $\eta_h$ in comparison to the estimator $\tilde\eta_h$ (right)
for the example with a discontinuous permeability and a quasi-uniform mesh with 96 elements.}
\label{fig:VarCoeff:robustness}
\end{figure}
For the last example, we choose a discontinuous permeability 
\[
\mu(x,y,z) = \left\{
\begin{array}{ll}
\mu_1 & \text{if } y<1/2 \text{ and } z < 1/2,\\
\mu_2 & \text{otherwise},
\end{array}
\right.
\]
on the unit cube $\Omega = (0,1)^3$,
and the right hand side $\mathbf{j}=(1,0,0)^t$.
We choose $k=2$, $\mu_1=1$, and vary $\mu_2= 10^\ell$ for $\ell=1,2,3$.
Since the exact solution is unknown, we approximate the error by comparing the numerical approximations
to a reference solution, which is obtained from the last numerical approximation by 8 more adaptive mesh refinements.
In this example, the adaptive algorithm refines strongly along the
edge with endpoints $(0,1/2,1/2)^t$ and $(1,1/2,1/2)^t$, similarly to what is shown in~\cite[Figure~6]{gedicke20}.
The errors in Figure~\ref{fig:VarCoeff} converge with about $\mathcal{O}((N_h/\ln(N_h))^{-2/3})$,
which is optimal for isotropic adaptive mesh refinement,
and the efficiency indices are robust with respect to the contrast of the permeability.

We also compare the efficiency indices of $\eta_h$ to those of $\tilde{\eta}_h$ for different polynomial degrees $k$. To obtain a reference solution, we take the numerical approximation and apply one uniform mesh refinement with respect to $h$. The results are shown in Figure \ref{fig:VarCoeff:robustness}. We observe again that the error estimator $\eta_h$ is robust with respect to the polynomial degree $k$, whereas the estimator $\tilde\eta_k$ grows when increasing $k$.

\section{Conclusions}\label{sec:conclusion}
We have introduced and analyzed an \emph{a posteriori}
 error estimator for arbitrary-degree N\'ed\'elec discretizations of
 the magnetostatic problem based on an equilibration principle.
 This estimator is constructed by adding a localized gradient correction to the estimator
 introduced in~\cite{gedicke20}, and is proven to be reliable
 with reliability constant 1, and uniformly efficient, not only in the
 mesh size, but also in the degree of the polynomial approximation.
 The computation of the new gradient term requires solving local
 problems on vertex patches.
 The polynomial-degree robustness of the
 new estimator has been numerically demonstrated on test problems
 with smooth as well as singular solutions, and for problems with a discontinuous magnetic permeability.

\bibliographystyle{abbrv}
\bibliography{PosterioriError,Maxwell3dNumerics2}

\appendix
\section{Proof of Proposition~\ref{prp:phiProjErr}}\label{appA}
Let $u_0\in P_0^{-1}(\Tnu)$ denote the $L^2$ projection of $u$ onto $P_0^{-1}(\Tnu)$, let $u_{0,T}$, for each $T\in\Tnu$, denote the value of $u_0$ at $T$, and let $\lambda_{0,f}:=u_{0,T^+}-u_{0,T^-}$ for each $f\subset\omega_\nu$. Note that $u_{0,T}-\overline{u}^{\omega_\nu}=0$ for all $T\in\Tnu$ when $\lambda_{0,f}=0$ for all $f\subset\omega_\nu$. We therefore have 
\begin{align*}
\left( \sum_{T\in\Tnu} \left(u_{0,T}-\overline{u}^{\omega_\nu}\right)^2 \right)^{1/2} &\leq \left( \sum_{f\subset\omega_\nu}\lambda_{0,f}^2 \right)^{1/2},
\end{align*}
where $C$ is some positive constant that only depends on the configuration of the element patch. Since the number of possible configurations is finite and depends on the mesh regularity, the constant $C$ only depends on the mesh regularity. From this inequality, we can obtain
\begin{align}
\label{eq:phiProjErr1a}
\|u_{0}-\overline{u}^{\omega_\nu}\|_{\omega_\nu} &\leq C\sum_{f\subset\omega_\nu} h_f^{1/2}\left\|{\ju{u_0}} \right\|_f.
\end{align}
We can also derive the following:
\begin{align*}
\|\ju{u_0}\|_f^2 &= (\ju{u_0},\ju{u_0})_f = (\ju{u_0-u},\ju{u_0})_f \leq \|\ju{u_0-u}\|_f\|\ju{u_0}\|_f
\end{align*}
for every $f\subset\omega_\nu$, where the second identity follows from the property $(\ju{u},1)_f=0$ and the third identity follows from the Cauchy--Schwarz inequality. From this, we obtain
\begin{align}
\label{eq:phiProjErr1b}
\|\ju{u_0}\|_f \leq \|\ju{u_0-u}\|_f \leq C\left(h_{T^+}^{1/2}\|\nabla u\|_{T^+} + h_{T^-}^{1/2}\|\nabla u\|_{T^-}\right) 
\end{align}
for all $f\subset\omega_\nu$, where the second inequality follows from standard interpolation theory. From interpolation theory, it also follows that
\begin{align}
\label{eq:phiProjErr1c}
\|u-u_0\|_T &\leq Ch_T\|\nabla u\|_T
\end{align}
The proposition then follows immediately from the triangle inequality, \eqref{eq:phiProjErr1a}, \eqref{eq:phiProjErr1b}, and \eqref{eq:phiProjErr1c}.

\end{document}